\documentclass{amsart}
  \usepackage[top=1in, bottom=1in, left=1.375in, right=1.375in]{geometry}
 \usepackage{amsmath,amssymb}
 \usepackage{algpseudocode}
 \usepackage{algorithm}
 \usepackage{algorithmicx}
 \usepackage{caption}
 \usepackage{subcaption}
 \usepackage{amsthm}
 \numberwithin{equation}{section}
 \usepackage{amsfonts}
 \usepackage{graphicx}
 \usepackage{hyperref}
 \usepackage{makecell}
 \usepackage{longtable}

\usepackage{amsmath}
\usepackage{amssymb}
\usepackage{amsthm}
\usepackage{amsmath}
\usepackage{setspace}
\usepackage{xcolor}
\usepackage{fancyhdr}
\usepackage{amssymb}
\usepackage{amsthm}
\usepackage{listings}
    \usepackage{cite}
    \usepackage{tabularx}
    \usepackage{graphicx}
    \usepackage{epstopdf}
    \usepackage{epsfig}
    \usepackage{float}
    \usepackage{ listings}
    \usepackage{appendix}
 \theoremstyle{plain}
 \newtheorem{thm}{Theorem}[section]
 \newtheorem{prop}[thm]{Proposition}
 \newtheorem{lem}[thm]{Lemma}
 
 \theoremstyle{definition}

 \theoremstyle{remark}
 \newtheorem{remark}[thm]{Remark}

 \let\pa=\partial
 \let\al=\alpha
 \let\b=\beta
 \let\d=\delta
 \let\g=\gamma
 \let\e=\varepsilon

 \let\lam=\lambda
 
 \let\s=\sigma
 \let\f=\frac
 \let \les = \lesssim
  \let \gtr = \gtrsim
 \let\om=\omega
 
 \let \th = \theta

 \let\G= \Gamma
\let\B = \Big
 \let\D=\Delta

 \let\td = \tilde
 \let\wt=\widetilde

 \let\teq \triangleq
 
 \let\pa=\partial

 \def\cE{{\mathcal E}}

 \def\cL{{\mathcal L}}

 \def\cR{{\mathcal R}}

 \def\na{\nabla}
 \def\la{\langle}
 \def\ra{\rangle}

\def\one{\mathbf{1}}

 \newcommand{\beq}{\begin{equation}}
 \newcommand{\eeq}{\end{equation}}
  \newcommand{\bal}{\begin{aligned} }
  \newcommand{\eal}{\end{aligned}}
 \newcommand{\ben}{\begin{eqnarray}}
 \newcommand{\een}{\end{eqnarray}}
 \newcommand{\beno}{\begin{eqnarray*}}
 \newcommand{\eeno}{\end{eqnarray*}}

 \newcommand{\uu}{\mathbf{u}}

\newcommand{\Z}{\mathbb{Z}}
\newcommand{\R}{\mathbb{R}}

\makeatletter
\newsavebox{\@brx}
\newcommand{\llangle}[1][]{\savebox{\@brx}{\(\m@th{#1\langle}\)}%
  \mathopen{\copy\@brx\kern-0.5\wd\@brx\usebox{\@brx}}}
\newcommand{\rrangle}[1][]{\savebox{\@brx}{\(\m@th{#1\rangle}\)}%
  \mathclose{\copy\@brx\kern-0.5\wd\@brx\usebox{\@brx}}}
\makeatother

 \author{Jiajie Chen}
 \date{ \today}

\title[Blowup of perturbed Landau equation]{Nearly self-similar blowup of the slightly perturbed homogeneous Landau equation with very soft potentials}
 
\begin{document}

\begin{abstract}

We study the slightly perturbed homogeneous Landau equation 
\[
\partial_t f = a_{ij}(f)  \cdot \partial_{ij} f + \alpha c(f) f, \quad c(f) = - \partial_{ij} a_{ij}(f),
\]
with very soft potentials, where we increase the nonlinearity from $ c(f) f$ in the Landau equation to $\alpha c(f) f$ with $\alpha>1$. For $\alpha > 1 $ and close to $1$, we establish finite time nearly self-similar blowup from some smooth initial data $f_0 \geq 0$, which can be both radially symmetric or non-radially symmetric. The blowup results are sharp as the homogeneous Landau equation $(\alpha=1)$ is globally well-posed, which was established recently by Guillen and Silvestre. To prove the blowup results, we build on our previous framework \cite{chen2020slightly,chen2021regularity} on sharp blowup results of the De Gregorio model with nearly self-similar singularity to overcome the diffusion. Our results shed light on potential singularity formation in the inhomogeneous setting.

\end{abstract}

 \maketitle

\section{Introduction}\label{sec:intro}

Whether the Landau-Coulomb equation,  proposed by Landau in 1936, can develop a finite time singularity is an important open problem in kinetic equations \cite{silvestre2023regularity,villani2002review}. 
The Landau equation with general potentials can be written as 
\beq\label{eq:LC}
\bal
 & \pa_t f + v \cdot \na_x f = Q(f, f)= a_{ij}(f) \pa_{ij} f + c(f) f,  \\
&  a_{ij}(f)  =  \phi^{ij}(v) \ast f , \quad c(f) = - \pa_{ij} a_{ij}(f), 
\quad \phi^{ij}(v) \teq   \f{1}{8 \pi} ( \d_{ij} - \f{v_i v_j}{|v|^2} ) |v|^{\g + 2}, 
\eal
\eeq
where $\g \in [-3, 1]$, $f( x, v, t) : \R^3 \times \R^3 \times [0, T] \to \R_+$ is a distribution function for the velocity $v$ at $x$, and $Q$ is the collision operator only acting on the velocity variable.
It can be reformulated as the divergence form  
\beq\label{eq:Q}
Q(f, g) =  \pa_i \int_{\R^3} \phi^{ij}(v - \td v ) \B( f( \td v) \pa_j g(v) - g(v) \pa_j f(\td v)   \B) d \td v   . 
\eeq
In the case of  $\g=-3$, we obtain the Coulomb potential and can simplify $a_{ij}, c$ as $a_{ij}(f) = - \pa_{ij}(-\D)^{-2}f(v) , c(f) = f $.

One of the major difficulties in establishing  global regularity of \eqref{eq:LC} arises from the competing effects of the diffusion $a_{ij}(f) \pa_{ij} f$ with nonlocal coefficients and the nonlinearity $c(f) f$, potentially leading to blowup. 
Both nonlinear terms have the same scaling, making it challenging to determine which term is stronger. 
Very recently, in a remarkable work of Guillen-Silvestre \cite{Lius2023Landau}, they 
 established the global regularity of homogeneous Landau equation \eqref{eq:LC} with interaction potentials covering \eqref{eq:LC} with $\g \in [-3, 1]$ by proving the monotonicity of the Fisher information in time. 
 For more discussions on the  regularity of the Landau equation, we refer to reviews by Villani \cite{villani2002review} and Silvestre \cite{silvestre2023regularity}. 





To study potential singularity formation of \eqref{eq:LC}, in this paper, we consider the slightly perturbed homogeneous Landau equation \eqref{eq:LC} with very soft potentials $\g \in [-3, -2)$ by increasing the strength of the nonlinearity  $c(f) f$
\beq\label{eq:LC_a}
\pa_t f(v) = Q(f, f) + (\al-1) c(f) f = a_{ij}(f) \cdot \pa_{ij}(f) + \al c(f) f 
, \quad c(f) = - \pa_{ij} a_{ij}(f).
\eeq

Since the homogeneous Landau equation does not blowup, to construct a potential singularity of \eqref{eq:LC}, one needs to make use of the transport part. We mimic its effect using the extra nonlinearity $ (\al -1) c(f) f$ and consider the homogeneous setting, where $f$ is independent of $x$. 
Although the transport part $ v \cdot \na_x f $ \eqref{eq:LC} is linear in $f$, since the solution scales differently in $x, v$, $ v \cdot \na_x f $ and $c(f) f$ can have the same scaling. See \cite{bedrossian2022non} for the scaling symmetries of \eqref{eq:LC}.
At the macroscopic level, the imploding singularities in compressible fluids \cite{merle2022implosion1,merle2022implosion2} suggest that the transport term can have highly nonlinear effects. It may play a role as concentrating the particles at some point $x_*$ and further lead to an effect similar to $(\al-1) c(f) f$ at $x_*$ in \eqref{eq:LC_a}.

The local well-posedness of \eqref{eq:LC_a} is a consequence of Theorem 1.1 \cite{henderson2019local} in the homogeneous setting, where the authors proved the local well-posedness of the Landau equation without using the divergence form. The above model has the same spirit as the Krieger-Strain model \cite{krieger2012global} for Landau equation and the generalized Constantin-Lax-Majda model for 3D Euler \cite{OSW08,DG90}. These models capture competing effects between two nonlinearities. See more discussions in Sections \ref{sec:euler}, \eqref{eq:KS}. 
It seems that there is limited prior work on potential singularity formation in kinetic equations. Our second motivation is to shed some light on this aspect.



Our main result is the following.

\begin{thm}\label{thm:blowup}
For each $\g \in [-3, -2)$, there exists an absolute constant $\d_{\g} > 0 $ such that for any $\al \in (1, 1 + \d_{\g})$, the perturbed Landau equation \eqref{eq:LC_a} develops a nearly self-similar singularity in finite time $T > 0$ from some nonnegative initial data $f_0 \in C_c^{\infty}$ or $f_0 \in C^{\infty}$ with $ f_0 > 0, f_0 \exp(|v|^2)  \les 1$. The initial data is even in $v_i$ and can be both radially symmetric or non-radially symmetric. 
Moreover, the mass and energy $\int f |v|^k dv,k=0,2$ blow up at $T$.
\end{thm}


By nearly self-similar blowup, we mean that the blowup solution, under a suitably dynamic rescaling, is close to an approximate blowup profile. By performing higher order stability estimates, one can prove that the blowup is asymptotically self-similar and $|| f||_{L^{\infty}}$ blows up at $T$. Since these additional steps follow our previous works \cite{chen2019finite,chen2019finite2,chen2021HL,chen2020slightly}
we do not pursue them. See Remark \ref{rem:asymp} for more discussions. Note that the initial data is allowed to decay polynomially. See Section \ref{sec:blowup}.
One may compare the potential blowup among the Landau equation, the model \eqref{eq:LC_a}, and the nonlinear heat equation 
\beq\label{eq:heat}
 \pa_t = \D f + \al f^2, \quad \al > 0.
\eeq

Singularity formation of 
\eqref{eq:heat} has been extensively studied \cite{merle1998optimal,giga1985asymptotically}. A key difference between \eqref{eq:heat} and \eqref{eq:LC} or \eqref{eq:LC_a} is that the diffusion and nonlinearity scale differently. 
Thus, $\al$ can be normalized to $1$. 
Moreover, the coefficient $a_{ij}(f)$ of our blowup solution to \eqref{eq:LC_a} also blows up, significantly enhancing the dissipation. See \eqref{eq:SS_ansatz} for the approximate blowup scaling. 
The enhanced dissipation makes it more difficult to establish (potential) singularity formation in \eqref{eq:LC_a}, \eqref{eq:LC} compared to \eqref{eq:heat}. 
See \cite{silvestre2023regularity} for related discussions on enhanced dissipation.



Theorem \ref{thm:blowup} highlights that to obtain global regularity of the inhomogeneous Landau equation \eqref{eq:LC}, one needs to make essential use of the divergence form.


To prove Theorem \ref{thm:blowup}, we need stability estimates of $\cL_1 f = Q(f, \mu) + Q(\mu, f) , \mu = e^{-|v|^2}$ for $f$ decaying only polynomially. We will use the highly nontrivial stability estimate of $\cL_1$ developed in \cite{carrapatoso2017landau}, which is based on semigroup. Since the case of Coulomb potential is the most physically important, we establish a new coercive estimate in this case and use it to prove blowup. 
We remark that we do not need estimates from \cite{carrapatoso2017landau} to prove blowup of \eqref{eq:LC_a} with Coulomb potential. 

\begin{thm}\label{thm:coer_poly}
Let $\mu = e^{- |v|^2}, \la v \ra = (1 + |v|^2)^{ \f{1}{2}}$, and $Q$ be the Coulomb collision operator $ \g = -3$ \eqref{eq:LC}, \eqref{eq:Q}. There exists a weight $W$ with $1 \les W(v) \les \la v \ra^{ \f{21}{2}}$ and an absolute constant $c_* > 0$ such that for any radially symmetric function $f$ orthogonal to $1, v_i, |v|^2, i=1,2,3$, we have
\beq\label{eq:thm_coer_poly}
 \int_{\R^3} (Q( f, \mu) + Q(\mu, f)) f W d v  \leq - c_* \int_{\R^3} \la v \ra^{-3} ( f^2 + |\na f|^2) W d v.
\eeq
\end{thm}

We will construct the weight $W$ explicitly. It is well-known that the linearized operator $\cL_1 f = Q(f, \mu) + Q(\mu, f) $ is coercive in $L^2(\mu^{-1/2})$ \cite{guo2002landau,hinton1983collisional,carrapatoso2017landau} with $\mu^{-1/2}$ growing exponentially. 
Yet, this coercive estimate is not sufficient for our purpose. It appears that 
\eqref{eq:thm_coer_poly} is the first coercive estimate of $\cL_1$ in some $L^2(W^{1/2})$ space with $W$ growing polynomially. The estimate \eqref{eq:thm_coer_poly} and its proof may reveal some delicate properties of  $\cL_1$ useful for other studies. The coercive estimates may be generalized to other potentials or without radial symmetry using the ideas and argument in Section \ref{sec:lin}, though the analysis would be more complicated. We refer more discussion of the estimate in \cite{carrapatoso2017landau} and Theorem \ref{thm:coer_poly} in Section \ref{sec:intro_lin}.

\subsection{Other related works and discussions of the model} We discuss some related works. 

\subsubsection{Regularity of the Landau equation}


The literature on kinetic equation is too vast for an exhaustive discussion, but we highlight some relevant results  on Landau equation with soft potential. The local existence 
\cite{AP1977exist,villani1998new} and uniqueness were established in \cite{fournier2009well,fournier2010uniqueness,henderson2019local,desvillettes2000spatially}. 

The global nonlinear stability of Maxwellians on a periodic box was first established by Guo \cite{guo2002landau}. Later, Guo and Strain \cite{strain2006almost,strain2008exponential} proved stretched exponential convergence to Maxwellians. 
See \cite{carrapatoso2016cauchy,luk2019stability,carrapatoso2017landau} for other recent results on stability near Maxwellians or vacuum. 


Finite time blowup of a model of Boltzmann equation without the loss term was established 
in \cite{andreasson2004blowup}. We refer to the excellent surveys \cite{silvestre2023regularity,villani2002review} for more discussions on 
kinetic equations.

\subsubsection{3D incompressible Euler equations and models}\label{sec:euler}

The Landau equation \eqref{eq:LC} has some structures similar to the 3D incompressible Euler equations. In the vorticity formulation, the Euler equations can be written as follows 
\beq\label{eq:Euler}
\om_t = - \uu \cdot \na \om  + \om \cdot \na  \uu = - \na \times ( \uu \cdot \na \uu),
\quad \uu = \na \times (-\D)^{-1} \om,
\eeq
where $\uu$ is the velocity and $\om = \na \times \uu$ is the vorticity. 
It is commonly believed that the vortex stretching $\om \cdot \na \uu$ is the driving force for finite time blowup. Yet, the advection term can regularize the solution \cite{lei2009stabilizing,hou2008dynamic,hou2006dynamic}. The above divergence structure and competing nonlinear effects 
share some similarities with the Landau equation \eqref{eq:LC}. 



Recent advances on singularity formation in \eqref{eq:Euler} have been made by choosing low regularity data \cite{elgindi2019finite,chen2019finite2,elgindi2019stability} or imposing a boundary for smooth data \cite{ChenHou2023a,chen2019finite2,ChenHou2023b,luo2013potentially-2} to weaken the advection. 
For the inhomogeneous Landau equation, the transport term may enhance the nonlinear effects. See the discussion below \eqref{eq:LC_a}.


By enhancing the strength of the vortex stretching \cite{chen2020singularity} or weaking the advection using $C^{\al}$ data \cite{chen2021regularity}, we proved sharp blowup results of the De Gregorio model \cite{DG90} of 3D Euler on a circle. The proof of Theorem \ref{thm:blowup} builds on the method in \cite{chen2020singularity,chen2021regularity}. See Section \ref{sec:ideas}. Thus, our blowup argument for 
\eqref{eq:LC_a} may shed light on the singularity formation of the original equation \eqref{eq:LC}.



\vspace{0.1in}

\paragraph{\bf{Conservation laws}}

One of the drawbacks of the model \eqref{eq:LC_a} with $\al >1$ may be that it does not conserve the mass $\int f d v$ and the energy $\int f |v|^2 d v$ like the Landau equation \eqref{eq:LC}. Yet, in  \eqref{eq:LC}, these a-priori conservation laws provide much weaker control on the solution in smaller scale, while two competing nonlinear terms remain the same order in smaller scale. In 3D Euler equations, the conservation of energy $|| \uu ||_{L^2}^2 $ also provides much weaker control on the solution in smaller scale. It does not prevent the blowup \cite{elgindi2019finite,chen2019finite2,ChenHou2023a}. 
In addition, in the inhomogeneous setting of \eqref{eq:LC}, mass or energy at a specific point $x$ is not conserved.
If the equation develops a singularity at $x_*$, it is natural to expect that these macroscopic
quantities at $x_*$ are strictly increasing in time near the blowup time or even blow up. We remark that continuation criterion based on the macroscopic quantities, mass, energy, and entropy, has been established for the Landau and Boltzmann equation without cutoff and with moderate soft potential \cite{imbert2022global,cameron2018global,golse2016harnack,henderson2019local,henderson2020c}. 





\vspace{0.1in}
\paragraph{\bf{Krieger-Strain model}}

There is a model problem for homogeneous Landau equation \eqref{eq:LC}, which closely relates to \eqref{eq:LC_a}. In \cite{krieger2012global}, Krieger and Strain introduced the following model
\beq\label{eq:KS}
\pa_t f = (-\D)^{-1} f \cdot \D f  + \al f^2
\eeq
with $\al \geq 0$, which replaces the diffusion $a_{ij}(f) \pa_{ij}(f)$ in \eqref{eq:LC} by $(-\D)^{-1} f \cdot \D f$.  

Global regularity of \eqref{eq:KS} for radially symmetric, positive and monotone data was established for $\al \in [0,1]$ \cite{krieger2012global,gressman2012non,gualdani2016estimates}.
The generalization of the model \eqref{eq:KS} was studied in \cite{gualdani2022hardy}.



\subsubsection{Approximately self-similar blowup}\label{sec:intro_rate}

Our blowup solution for \eqref{eq:LC_a} is nearly (approximately) self-similar and can be further proved to be asymptotically self-similar. 
There has been some progress in ruling out type I self-similar blowup solution to 
the Landau equation \eqref{eq:LC} and Boltzmann equation under some decay assumptions \cite{bedrossian2022non}. We note that an important class of type I self-similar blowup to \eqref{eq:LC} has not been ruled out by \cite{bedrossian2022non}, and it is meaningful to study (approximately) self-similar blowup of \eqref{eq:LC} and related models. 

We focus on the homogeneous Landau equation \eqref{eq:LC}. Although it is globally regular \cite{Lius2023Landau}, we use it to illustrate the issue for simplicity. The homogeneous self-similar blowup solution to 
\eqref{eq:LC} $f(t, v ) = \f{1}{T-t} F( \f{v}{(T-t)^{\th}} )$ satisfies $F \geq 0$ and 
\beq\label{eq:profile}
 \th v \cdot \na F + F = Q( F, F ).
\eeq
In Theorem 2.5 in \cite{bedrossian2022non}, the authors proved that if $F \in C^{\infty}$ and $F \in L^1$ for $\th \in  ( \f{1}{3}, \f{1}{2}) $ or $F (1 + |v|^2) \in L^1$ for $\th = \f{1}{3}$, $F$ must be $0$. We note that if $F$ decays in the far-field, e.g. $| F| \les |v|^{-k}$ for some $k > 0$,  the nonlinear term $Q(F, F)$ decays faster. 
Balancing the decay rates for the linear terms in \eqref{eq:profile} for large $|v|$, one gets the natural decay rate $ F \sim |v|^{-1/\th}$, which does not satisfy the above integrable assumption for $F$ and $\th \in [1/3, 1/2]$. Thus, this class of blowup profile is not ruled out by \cite{bedrossian2022non}. Similar argument generalizes to the inhomogeneous setting.

We note that (approximately) self-similar blowup with slowly decaying profiles are very common in fluid PDEs with two-parameter scaling symmetry groups. Examples include Burgers' equation \cite{collot2018singularity}, 3D incompressible Euler and 2D Boussinesq equations \cite{chen2019finite2,ChenHou2023a,ChenHou2023b,elgindi2019finite}, and related models \cite{chen2021HL,chen2019finite}. In these examples, decay rates of the profiles are determined by balancing the linear parts similar to the above, and the velocity can grow sublinearly in the far-field. 
The approximately self-similar blowup solution 
has finite energy by truncating the profile  \cite{chen2019finite2,ChenHou2023a,ChenHou2023b,elgindi2019stability}.

\subsection{Ideas of the proof}\label{sec:ideas} Below, we discuss the ideas in the proof. 

\subsubsection{Blowup solution based on the Maxwellian}\label{sec:idea_ASS}

One of the major difficulties for singularity formation of \eqref{eq:LC_a} is to overcome the diffusion \eqref{eq:LC_a}. 
We adopt the methods in our previous works on incompressible fluids \cite{chen2019finite,ChenHou2023a,ChenHou2023b,chen2019finite2} and the ideas in \cite{chen2020slightly,chen2021regularity} to construct a blowup solution by perturbing the steady state. Our goal is to construct a nearly self-similar blowup solution close to the following form 
\beq\label{eq:SS_ansatz}
 f(t, v) = 
(T-t)^{c_{\om}} F( \f{v}{(T-t)^{c_l}}),
\eeq
where $F$ is close to the Maxwellian $\mu = \exp(-|v|^2)$ and has a decay rate $|v|^{-\lam}, \lam 
= |\f{ c_{\om}}{c_l}| > 5$. See Lemma \ref{lem:mono}. The scaling of $f$ 
and the estimate in \eqref{eq:blowup_vk} imply the blowup of $a_{ij}(f)$ \eqref{eq:LC_a}.  We study  \eqref{eq:LC_a} using the dynamic rescaling equation 
\beq\label{eq:res_intro}
\pa_t f + c_l v \cdot \na f = c_{\om} f + Q(f, f) + (\al -1 ) f^2. 
\eeq

To develop stability estimates later, we need to impose orthogonal conditions: $\int  f(t) g d v $ is constant in $t$ for $g = 1, v_i, |v|^2$. Different from the homogeneous Landau equation \eqref{eq:LC}, 
\eqref{eq:res_intro} does not preserve these five conditions. 
A key idea is to impose orthogonality with $1, |v|^2$ by choosing normalization conditions $ \int f(t) |v|^k = \int f_0 |v|^k d v, k= 0, 2$ for $c_l, c_{\om}$. Moreover, by considering $f$ even in $v_i$, which is preserved by the equation, we obtain other three conditions.
See Section \ref{sec:ASS} for the formulas of $c_l ,c_{\om}$ and more details on the dynamic rescaling reformulation. 

Formally $\al$ characterizes the relative strength between the diffusion and nonlinear terms $\al c(f) f$. A crucial observation is that if $\al >1$ and $f$ is close to the Maxwellian $ \mu = \exp(- |v|^2)$ for all time $t>0$ in \eqref{eq:res_intro}, the normalization conditions imply that $c_l \asymp \al -1 > 0, c_{\om} \asymp -(\al-1) < 0$ for  $ t> 0$. 
Although $|\al-1|$ is extremely small, these sign conditions allow us to obtain finite time blowup
using a rescaling argument. 
We will construct the approximate steady state $\bar f = \exp(-|v|^2)$, $\bar c_l \asymp (\al-1), \bar c_{\om} \asymp -(\al-1) $ and prove its nonlinear stability in \eqref{eq:res_intro} with 
$|| f(t) - \mu ||_X \les \al-1$ for all $t>0$ in suitable norm $X$. 



\subsubsection{Linear stability analysis and a new coercive estimate}\label{sec:intro_lin}

In the blowup analysis, we need to study the linearized equation of \eqref{eq:res_intro} around the approximate steady state $(\bar f, \bar c_l, \bar c_{\om})$
\beq\label{eq:lin_intro}
\pa_t f = \cL_{\al} f 
\teq  - \bar c_l v \cdot \na f + \bar c_{\om} f + Q(\bar f, f) + Q(f, \bar f) + (\al-1) ( c(\bar f) f + c(f) \bar f) .
\eeq
Here, $f$ denotes the perturbation.
Several nonlinear stability estimates of the Maxwellian have been established based on the $L^2(\mu^{-1/2})$ coercive estimate of $\cL_1 $ \cite{guo2002landau,hinton1983collisional,carrapatoso2017landau}: 
$\la \cL_1 f , f \mu^{-1} \ra \leq 0.$
 Yet, for $0< \al-1$ very small, we cannot treat $\cL_{\al}$ as a small perturbation to $\cL_1$ in $L^2(\mu^{-1/2})$ estimate. 
In fact, in $L^2(\mu^{-1/2})$ estimate, the scaling term $-\bar c_l v \cdot \na $ contributes to 
\[
 - \int \bar c_l v \cdot \na  f \cdot f \mu^{-1} d v 
 = \f{\bar c_l}{2} \int \na \cdot ( \mu^{-1} v ) f^2 d v
 = \f{\bar c_l}{2} \int (3 + 2 |v|^2)  f^2 \mu^{-1} d v ,
 \]
 which cannot be bounded by the energy norm $||f||_{L^2(\mu^{-1/2})}$. Similarly, if  $\rho$ in $L^2(\rho)$ estimate of $\cL_1$ grows too fast, e.g. $\rho(v) \gtr |v|^{k}$ with $k$ very large, 
 $-\bar c_l v \cdot \na $ leads to a term that destroys the stability. 
For the most important case \eqref{eq:LC_a} with Coulomb potential, we will develop  
a new coercive estimate \eqref{eq:thm_coer_poly}.

In \cite{carrapatoso2017landau}, the authors established highly nontrivial stability estimates of $\cL_1$ in some norm equivalent to $L^2( \la v \ra^{\th})$ with $\th > \f{7}{2}$, using an enlargement trick and semigroup stability estimates. We will apply these estimates to obtain stability estimates of $\cL_1$ for non-radially symmetric solution or potentials other than Coulomb. Our coercive estimates for Coulomb are drastically different from \cite{carrapatoso2017landau} and are based on direct estimates in weighted $L^2$ space. Additionally, the estimates are more elementary than \cite{carrapatoso2017landau}. We outline our coercive estimates below.




\vspace{0.1in}
\paragraph{\bf{Coercive estimates in the near-field}}

We carefully design the weight $\rho$ with 
\[
\rho(v) = \mu^{-1} = \exp(|v|^2) , \ |v| \leq R_1, \quad \rho(v) \sim C |v|^k , \ r  \geq R_2, \ k > 2 
\]
for some $R_1 < R_2$ large enough, and then establish the coercive estimates up to a small loss: 
\beq\label{eq:coer_near}
\bal
  & \la - \cL_1 f, f \rho \ra_{\R^3} \geq Q_{\rho} - o(R_1^{-\e}) || f \la v \ra^{k_2}||_{L^2}^2, \ k_2 > 0, \  \e > 0 \\
 & Q_{\rho} = \f{1}{2} \int_{\R^3} \int_{\R^3}  M_{\rho}(v, \td v) ( F_{\rho, i}(v) - F_{\rho,i}(\td v) )( F_{\rho,j}(v) - F_{\rho,j}(\td v)) \phi^{ij}(v - \td v) d v d \td v  ,
  \eal
\eeq
where  $ M_{\rho} \geq 0$ satisfies $M_{\rho}(v, \td v) = M_{\rho}(\td v , v)$, and $F_{\rho,i}$ depends on $f, \rho$. Since $ \{ \phi^{ij} \}_{ij}$ \eqref{eq:Q} is positive semi-definite and $M_{\rho} \geq 0$, we get $Q_{\rho } \geq 0$.

The above estimates are the most delicate parts 
and involve symmetrizing and comparing several integrals. We mainly use the radial symmetry of the solution to establish the above estimate. Note that gluing weights $\exp(|v|^2)$ and $|v|^k$ using some cutoff functions directly leads to a large loss term, which can destroy the stability estimates.

\vspace{0.1in}
\paragraph{\bf{Localized coercive estimates}}  Since $\rho = \mu^{-1}$ for $|v| \leq R_1$, for $|v|, |\td v| \leq R_1$, the integrand in $Q_{\rho}$ is the same as that in the coercive estimate \cite{guo2002landau}. We generalize the compactness argument in \cite{guo2002landau} to establish the coercive estimate 
\[
Q_{\rho}(f) \geq 
\d_n |  f \mu^{-1/2}|_{\s, n }^2, \quad \int_{|v| \leq n} f |v|^k d v = 0 , \ k=0,2, 
\quad \liminf_{n \to \infty} \d_n \geq \min(1, \d_{\infty}),
\]
for radially symmetric functions $f$, where $| g|_{\s, R_1}$ is some localized norm. In particular, $\d_n \geq \d_* > 0$  uniformly for some constant $\d_*$. See Lemma \ref{lem:coer_unif} for more details. The above localized coercive estimates appear to be new.

For large $|v|$, 
we treat the nonlocal terms in $\cL_1$ perturbatively (see \eqref{eq:lin}, \eqref{eq:lin_terms_rad}) since their coefficients decay exponentially fast,  much faster than those of the local parts.  We obtain dissipative estimates of $\cL_1$ in $L^2(\rho_2)$ up to an $O(1)$ loss in the near field with some weight $\rho_2$ growing polynomially using integration by parts. 
A similar property has also been used in \cite{carrapatoso2017landau,carrapatoso2015rate}.



By choosing  suitable polynomial growth rates for $\rho, \rho_2$, the extra terms $- \bar c_l v \cdot \na f + \bar c_{\om} f $ in \eqref{eq:lin_intro} contribute to damping terms in the energy estimates, and we have a spectral gap for $\cL_{\al}$ of order $\al-1$. Passing from linear stability estimate to nonlinear stability and finite time blowup are relatively standard \cite{chen2019finite,chen2019finite2}.

\vspace{0.1in}
\paragraph{\bf{Organization of the paper}}
The rest of paper is organized as follows. In Section \ref{sec:ASS}, we perform the dynamic rescaling reformulation of \eqref{eq:LC_a}, construct the approximate steady state, and derive the linearized equation. Section \ref{sec:lin} is devoted to a new coercive estimate and the linear stability analysis 
in the case of Coulomb potential with radial symmetry. We consider linear stability analysis for  general cases in Section \ref{sec:lin_nonrad}, and prove nonlinear stability and finite time blowup in Section \ref{sec:non}.
Some derivations about radially symmetric functions are given in the Appendix.

\section{The dynamic rescaling formulation and the approximate steady state}\label{sec:ASS}

In this section, we reformulate the problem using the dynamic rescaling equation. Then we construct an approximate steady state and derive the linearized equation. Throughout the paper, we consider solution $f$ to \eqref{eq:LC_a} even in $v_i$ so that $\int f v_i d v = 0, i=1,2,3$.


\subsection{Dynamic rescaling formulation} 

The solution to \eqref{eq:LC_a} satisfies the following scaling symmetry. 
If $f$ is a solution to \eqref{eq:LC_a}, for any $\lam ,\nu > 0$,
\[
f_{\lam, \nu}(v, t) = \nu \lam^{ 3 + \g} f( \lam v , \nu t )
\]
is another solution to \eqref{eq:LC_a}. Let $ f(v, t)$ be the solution to the physical equation \eqref{eq:LC_a}. It is easy to show that 
\beq\label{eq:dyn1}
\bal
 & F(v, \tau ) = C_{\om}(\tau)  f( C_l(\tau) v, t(\tau)) ,
 \eal
\eeq
is the solution to the dynamic rescaling equation 
\[
 \pa_{\tau} F + c_l v \cdot \na F = Q(F, F)  + (\al-1) c(F) F + c_{\om} F, 
 \]
where 
\beq\label{eq:dyn2}
  C_l( \tau ) = \exp( \int_0^{\tau} -c_l(s) ds ), \quad   C_{\om}(\tau) = \exp( \int_0^{\tau} c_{\om}(s) ds ) , \quad t(\tau) = \int_0^{\tau} C_{\om}(s)  C_l(s)^{-\g-3} d s. 
\eeq

We have the freedom to choose the time-dependent scaling parameters $c_l(\tau), c_{\om}(\tau)$ according to some normalization conditions. After we choose them, 
the above rescaling equation is completely determined and it is equivalent to \eqref{eq:LC_a} using the above scaling relation. 


The dynamic rescaling formulation was introduced in \cite{mclaughlin1986focusing,  landman1988rate} to study the self-similar blowup of the nonlinear Schr\"odinger equations. This formulation is closely related to the modulation technique in the literature. 
It has been a very effective tool to analyze singularity formation for many problems like the nonlinear Schr\"odinger equation \cite{kenig2006global,merle2005blow}, compressible fluids \cite{buckmaster2019formation,buckmaster2019formation2,merle2022implosion1,merle2022implosion2}, the nonlinear wave equation \cite{merle2015stability}, the nonlinear heat equation \cite{merle1997stability}, 
and other dispersive problems \cite{martel2014blow}. Recently, this method has been applied to study singularity formation in incompressible fluids with $C^{1,\al}$ data \cite{chen2019finite2,elgindi2019finite} and with smooth data \cite{ChenHou2023a,ChenHou2023b} and related models \cite{chen2019finite,chen2020slightly,chen2021regularity,chen2020singularity,chen2021HL}.

If there exists $C>0$ such that   $c_{\om}(\tau) \leq -C <0$ and $F$ is nontrivial, e.g. $ || F(\tau, \cdot) ||_{L^{\infty}} \geq C >0$ for all $\tau >0$, 
using a rescaling argument, we yield finite time blowup \cite{chen2019finite,chen2019finite2}. 

To simplify the notation, we use $t$ for $\tau$, and $f, g$ for $F, G$. With these simplifications, we can rewrite the above rescaling equation as follows 
\beq\label{eq:LC_a_rescal}
 \pa_{t} f + c_l v \cdot \na f 
 = Q(f, f) + (\al -1) c(f) f + c_{\om} f.
\eeq


\subsection{Approximate steady state and normalization conditions}


Following the ideas in \cite{chen2020slightly,chen2021regularity} and in Section \ref{sec:idea_ASS}, we construct the approximate steady state $\bar f(v) = \exp(-|v|^2)$.

\vspace{0.1in}
\paragraph{\bf{Normalization conditions}}
To determine the scaling parameters $c_l, c_{\om}$, we impose the following normalization conditions on the solution to \eqref{eq:LC_a_rescal}
\beq\label{eq:normal0}
 \int_{\R^3} f(t) |v|^k d v = \int_{\R^3} \mu |v|^k d v, \quad k = 0, 2. 
\eeq
and only consider initial data with $f_0 - \mu$ orthogonal to $1, |v|^2$.



Since $1, |v|^2 $ are the collision invariants $\int Q(f, f) |v|^k d v , k=0,2$, to impose \eqref{eq:normal0} in \eqref{eq:LC_a_rescal}, using the evolution of 
$\int f |v|^k d v $ and \eqref{eq:LC_a_rescal}, we obtain the following conditions for $c_l(t), c_{\om}(t)$, 
\[
- (k+ 3) c_l \int f |v|^k d v =  c_l \int v \cdot  \na f \cdot |v|^k d v = c_{\om} \int f |v|^k d v + (\al-1) \int  c(f) f |v|^k d v , \quad k = 0, 2.
\]
Solving the above equations, we derive the normalization conditions for $c_l, c_{\om}$
\beq\label{eq:normal}
\bal
& c_{\om} = \f{\al-1}{2} \B( - 5 \f{ \int c(f) f  }{ \int f  } + 3 \f{ \int c(f) f |v|^2 }{ \int f |v|^2 }   \B) \teq C_1(f) (\al-1),  \\
& c_l = \f{\al-1}{2} \B(    \f{ \int c(f) f  }{ \int f  } - \f{ \int  c(f) f |v|^2 }{ \int f |v|^2 }    \B) \teq C_2(f)(\al-1).
\eal
\eeq
where the above integrals are over $v\in \R^3$. We will estimate $C_1(f), C_2(f)$ in Lemma \ref{lem:mono}.

For $ \al>1$ with $\al-1$ sufficiently small, we will prove that $f = \bar f + \td f$ with perturbation $\td f$ small in some norm. As a result, we have $C_1(f)  \approx C_1(\bar f) < 0, C_2(f) \approx C_2(\bar f) > 0$. 
Although $|\al -1|$ is extremely small, we can obtain the crucial sign conditions: $c_l \asymp |\al -1|, c_{\om} \asymp -|\al-1|$ 
uniformly for all $t>0$. Then using a rescaling argument, we can obtain finite time blowup of \eqref{eq:LC_a}. An argument similar to the above has been developed by the author in \cite{chen2020slightly,chen2021regularity}.


\vspace{0.1in}
\paragraph{\bf{Linearized equation}}

We choose the following approximate steady state $(\bar f, \bar c_l, \bar c_{\om})$ for \eqref{eq:LC_a_rescal}
\beq\label{eq:ASS}
\bar f = \exp( -|v|^2 ), 
\quad \bar c_{\om} = C_1(\bar f)(\al-1), \quad \bar c_l = C_2(\bar f)(\al -1).
\eeq

Linearizing \eqref{eq:LC_a_rescal} around $(\bar f, \bar c_l, \bar c_{\om}))$, we yield the equation for the perturbation $f$
\beq\label{eq:lin}
\bal
& \pa_t f  = \cL_a f  + N(f) + N(\bar f), \\
& \cL_a = - \bar c_l v \cdot \na f + \bar c_{\om} f
+ Q(f, \bar f) + Q(\bar f, f) 
+ (\al - 1) ( c(\bar f) f + c(f ) \bar f),
 \eal
\eeq
where $c_l, c_{\om}$,  the nonlinear terms $N(f)$ and the error terms $N(\bar f)$ are given by
\beq\label{eq:lin_terms}
\bal
c_l & = (\al -1) ( C_2(f + \bar f) - C_2(\bar f) ), \quad c_{\om} = (\al -1) (C_1(f + \bar f) - C_1(\bar f)), \\
N(f) & =  - c_l v \cdot \na f + c_{\om} f + Q(f, f) + (\al -1) c(f) f , \\
N(\bar f)  & = - ( \bar c_l + c_l ) v \cdot \na \bar f + ( \bar c_{\om} + c_{\om}) \bar f + (\al -1) c(f) f ,
\eal
\eeq
where we have used $Q(\bar f, \bar f) = 0$ to simplify $N(\bar f)$, and put the linear part $c_l v \cdot \na \bar f, c_{\om} \bar f$ to $N(\bar f)$ since it is very small.

\vspace{0.1in}
\paragraph{\bf{Coulomb potential with radial symmetry}}
In the case of Coulomb potential with radially symmetric solution $f$, we have $Q(f, f) = -\pa_{ij}(-\D)^{-2} \pa_{ij} f + f^2 $, and can rewrite it  as follows 
\beq\label{eq:Q_1D}
  Q(f, f) = -  f_{rr}  g_{rr} - \f{2}{r^2} f_r g_r + f^2,
\quad g = (-\D)^{-2} f ,
\eeq
where the formulas of $g_r, g_{rr}$ 
are given in \eqref{eq:BSlaw2},\eqref{eq:BSlaw3}. We can rewrite $N(f), N(\bar f), \cL_{\al} $  as follows  
\beq\label{eq:lin_terms_rad}
\bal
 \cL_{\al} & =  - \bar c_l r \pa_r f  + \bar c_{\om} f  - f_{rr} \bar g_{rr} - \f{2}{r^2} f_r \bar g_r 
  - \bar f_{rr} g_{rr} - \f{2}{r^2} \bar f_r g_r + 2 \al \bar f f    ,  \\
N(f) & = -  c_l r \pa_r f  -  f_{rr} g_{rr} - \f{2}{r^2} f_r g_r  + \al f^2 + c_{\om} f, \\
 N(\bar f) &= -  (\bar c_l + c_l) r \pa_r \bar f  + (\al -1) \bar f^2 + ( \bar c_{\om} + c_{\om} ) \bar f .
\eal
\eeq


We have the following important inequalities about $\bar c_{\om}, \bar c_l$ and  $|\bar c_{\om}/ \bar c_l|$. 
We refer more discussion of the relation between $|\bar c_{\om}/ \bar c_l|$ and decay rates of the profile to Section \ref{sec:intro_rate}, and the relation between $|\bar c_{\om}/ \bar c_l|$ and energy estimates for stability to \eqref{eq:Ladd_gam} and Section \ref{sec:lin_nonrad}.

\begin{lem}\label{lem:mono}
Let $C_i(f),\bar c_l, \bar  c_{\om}$  be defined in \eqref{eq:normal}, \eqref{eq:ASS}. For $\g \in [-3, 0)$, we have 
\beq\label{eq:ASS0}
C_1(\bar f ) < 0, \quad C_2(\bar f ) > 0, \quad C_1(\bar f) + 5  C_2(\bar f) < 0,  \quad 
\bar c_{\om}(\g) < 0,  \ \bar c_l(\g) >0,  \ |\bar c_{\om} / \bar c_l| > 5. 
\eeq
In particular, for $\g = -3$, we have $ |C_1(\bar f ) / C_2(\bar f)|  = |\bar c_{\om} / \bar c_l| = 7$. 
\end{lem}

\begin{proof}

Denote 
\[
Q_k = \int c(\mu) \mu |v|^k d v , \quad  P_k = \int \mu |v|^k d v, \ k = 0 ,2  , 
\]

\textbf{Case 1: $\g = -3$.}  
Using a change of variable $v \to \sqrt{2} v$, for $k=0, 2$, we yield 
\[
\bal
 & 
\int \exp(-|v|^2) |v|^k d v 
= 2^{ (k+3)/2} \int \exp(- 2 |v|^2) |v|^k d v, \ 
  \f{Q_2}{P_2}= 2^{- \f{5}{2} }, \quad \f{Q_0}{P_0} = 2^{- \f{3}{2}}, \\
& C_1( \bar f)  =  \f{1}{2}( - 5 \f{Q_0}{P_0} + 3 \f{Q_2}{P_2} )
 = - \f{7}{8 \sqrt{2}} ,
\quad C_2(\bar f)  = \f{1}{2}( \f{Q_0}{P_0}  - \f{Q_2}{P_2}   )
 = \f{1}{8 \sqrt{2}}  .
 \eal
\]

The desired estimate follows.

\vspace{0.1in}
\textbf{Case 2: $\g \in (-3, 0)$.}  The proof is based on the following inequality 
\beq\label{eq:cl_cw_mon}
v \cdot \na   c(\mu)  \leq 0 , \quad \g \in (-3, 0), \quad \mu = \exp(-|v|^2),
\quad c(\mu) \teq c_{\g} |v|^{\g} \ast  \mu.
\eeq
for some $c_{\g} > 0$,  with equality only at $v= 0$. 
A direct computation yields 
\[
I = v \cdot \na   c(\mu)  =
c_{\g}\int |w|^{\g}  v \cdot \na_v \exp(-|v-w|^2) d w 
=  -2 c_{\g}  \int |w|^{\g  }  v \cdot (v-w) \exp(-|v-w|^2) d w 
\]

Using a change of variable $w \to v + w$ and then symmetrizing the integral, we get 
\[
I =  2 c_{\g} \int |v + w|^{\g} v \cdot w \exp( -|w|^2) d w
 = c_{\g} \int  (|v+w|^{\g} - |v-w|^{\g})  v \cdot w \exp( - |w|^2) d w.
\]
Since $|v+w|^2 - |v-w|^2 = 4 v \cdot w, \g < 0$, we derive $ (|v+w|^{\g} - |v-w|^{\g})  v \cdot w \leq 0$ and $I \leq 0$. Clearly, $I = 0$ if and only if $v = 0$.



Now, we are in a position to prove Lemma \ref{lem:mono} with $\g > -3$. Since $v  \cdot \na \mu = - 2 |v|^2 \mu$, Using integration by parts, we have 
\[
\bal
P_2 & = \int \exp(-|v|^2) |v|^2 d v 
= - \int \f{1}{2} v\cdot \na \mu d v = \f{1}{2} \int  \mu \na \cdot v  d v = \f{3}{2} P_0,  \\
Q_2 & = -\f{1}{2} \int c(\mu) v \cdot \na \mu  d v 
= \f{1}{2} \int \na \cdot ( c(\mu) v ) \mu d v
= \f{1}{2} \int ( (v \cdot \na c(\mu) ) \mu + 3 c(\mu) \mu ) d v .
\eal
\]
Using \eqref{eq:cl_cw_mon}, we obtain 
\[
Q_2 < \f{3}{2} \int c(\mu)  \mu d v = \f{3}{2} Q_0. 
\]
It follows 
\[
\f{Q_2}{Q_0} < \f{P_2}{P_0}, \quad  \f{Q_2}{P_2} < \f{Q_0}{P_0}  .
\]

Clearly, we have $P_k , Q_k > 0$. Hence, we obtain 
\[
C_1( \bar f, \g)  =  \f{1}{2}( - 5 \f{Q_0}{P_0} + 3 \f{Q_2}{P_2} ) < 0,  
\quad C_2(\bar f,\g)  = \f{1}{2}( \f{Q_0}{P_0}  - \f{Q_2}{P_2}   ) > 0,
\quad  C_1 + 5 C_2 
=  \f{1}{2}  \cdot (-2 \f{Q_2}{P_2} ) < 0.
\]
Combing the above estimates and the conditions \eqref{eq:normal}, \eqref{eq:ASS0}, we prove Lemma \ref{lem:mono}.
\end{proof}


\section{Linear stability for Coulomb potential with radial symmetry}\label{sec:lin}

In this section, we perform linear stability analysis of \eqref{eq:lin} with Coulomb potential 
$\g = -3$  and prove Theorem \ref{thm:coer_poly}. Throughout this section, we consider radially symmetric solution.


The main part of the linearized operator of $\cL_a$ \eqref{eq:lin}, \eqref{eq:Q_1D} is $\cL_1$ 
\beq\label{eq:lin_op}
\cL_1 =  - f_{rr} \bar g_{rr} - \f{2}{r^2} f_r \bar g_r 
- \bar f_{rr} g_{rr} - \f{2}{r^2} \bar f_r g_r +
2  \bar f f = Q(\bar f, f) + Q(f, \bar f).
\eeq
Our goal is to establish $L^2(W)$ coercive estimates of $\cL_1$ for some weight $W$ growing polynomially and $f$ radially symmetric.
We refer the ideas of the coercive estimates to Section \ref{sec:ideas} and \ref{sec:ansatz}.



\vspace{0.1in}
\paragraph{\bf{Notations}}

Denote $r = |v|$. By abusing the notations and to simplify the notations, for any radially symmetric function $f(v) = F(|v|), v\in \R^3$ for some $F$, we will write $f(r)$ for $F(r)$.  
We use $\pa_r$ for the radial derivative. Under these notations, when we write
\[
\pa_i f = \pa_i |v| \cdot \pa_r f  = (v_i / r) \pa_r f  
\]
it should be interpreted as $\pa_{v_i} f(v) = \f{v_i}{r} \pa_r F$. 
We will write similar identities for radially symmetric functions. 
Denote by $\la \cdot , \cdot \ra  $ the standard inner product 
in $\R^3$: $\la f , g \ra = \int_{\R^3} f g d v $.
For a radially symmetric function $f$ in $\R^3$, we have 
\beq\label{eq:int_rad}
\int_{\R^3} f(v)  d v
= 4 \pi \int_0^{\infty} f(r)  r^2 dr  .
\eeq

For Coulomb potential $\g = -3$ \eqref{eq:LC}, denote 
\beq\label{eq:nota1}
\bal
 &\phi^{ij}(v) \teq  \f{1}{8\pi}  (\d_{ij} - \f{v_i v_j}{|v|^2} ) \f{1}{|v|}, \quad
 a_{ij}(f) = -\pa_{ij} (-\D)^{-2} f, \quad c(f) = f, \\
 &   \mu(v) = e^{- |v|^2}, \quad 
   \s^{ij}(v) \teq \phi^{ij} \ast \mu  = - \pa_{ij} \bar g ,  \quad \bar g = (-\D)^{-2}  \mu.
   \eal
\eeq

\vspace{0.1in}
\paragraph{\bf{Weights }}

In the rest of the paper, we use $\rho(v), \rho_i(v)$ to denote radially symmetric weights. In this section, $ \lam,  q, \eta$ are intermediate variables related to a weight $\rho$
\beq\label{eq:nota2}
q \teq \f{\pa_r \rho}{r}, \quad \eta \teq \f{ q^{\prime}}{ 2 r q}, 
\quad \lam = \f{\rho_r}{\rho} = \f{q r}{\rho} .
\eeq
If $ \rho = e^{r^2}$, we yield $q = 2 e^{r^2} = 2 \rho, \eta = 1 , \lam = 2 r, 
\rho + \f{q}{2} - \f{\lam \rho}{r} = 0$.

\vspace{0.1in}
\paragraph{\bf{Some identities for Coulomb potential}}

Using \eqref{eq:deri}, we have the following useful identity 
\beq\label{eq:Gauss_id1}
\bal
& v_i v_j \s^{ij} = - v_i v_j \pa_{ij} \bar g 
= - v_i v_j \f{v_i v_j}{r^2} \bar g_{rr}
 = - r^2 \bar g_{rr} >0, \\
& (r^3 \bar g_{rr})_r = 3 r^2 \bar g_{rr} + r^3 \bar g_{rrr}
=  - r^2 B_1(\bar f) = - \f{1}{2} r^2 \bar f , 
 \eal
\eeq
where we have used \eqref{eq:BSlaw3} and 
\[
B_1 = \int_r^{\infty} e^{-r^2} r dr = -\f{1}{2} e^{-r^2} \B|_r^{\infty} =  \f{1}{2}  \bar f, 
\]
in the second identity. We remark that the second identity is a special property for Coulomb potential since $ (r^3 \bar g_{rr})_r$ decays much faster than $ r^3 \bar g_{rr} \sim O(1)$. 

\vspace{0.1in}
\paragraph{\bf{Operators and quadratic forms}}

Using the above notations, we can rewrite the collision operator \eqref{eq:Q} as follows 
\[
Q(f, g) = \pa_i \int_{\R^3} \phi^{ij}(v - \td v ) ( f( \td v) \pa_j g(v) - g(v) \pa_j f(\td v)   )
d \td v 
= \pa_i \B(  ( \phi^{ij}\ast f) \pa_j g  \B) -
\pa_i \B( g \cdot ( \phi^{ij}\ast \pa_j f)  \B)
\]

From Section 2 in \cite{guo2002landau}, linearizing $Q(f + \bar f, f +\bar f)$ around $\bar f$, 
we obtain the following formulas for the linearized operator $\cL_1$ \eqref{eq:lin_op}, its local part, and nonlocal part 
\beq\label{eq:lin_op2}
\bal
&\cL_1 f =   \cL_{loc} + \cL_{nloc},  \quad \cL_{loc} = Q(\bar f , f),  \quad  \cL_{nloc} =  Q(f, \bar f) , \\
& \cL_{loc}  f = \pa_j ( \s^{ij}( \pa_i f + 2 v_i f )  )  ,  
\quad   \cL_{nloc}  f 
= -  \pa_i ( \mu \phi^{ij} \ast ( \pa_j f + 2 v_j f ) ), 
\eal
\eeq
where $i,j$ sum over $1,2,3$. We do not study the normalized perturbation $g = \mu^{-1/2} f $ as \cite{guo2002landau}.

We will use $Q_{\rho}, J_{\rho}$ to denote some quadratic forms related to $Q(f, f)$, and $\td J_{\rho}, \td{\td{J_{\rho}}},  J_{\e, \rho}$ the modifications of $J_{\rho}$.

\subsection{Ansatz of coercive estimates and modified quadratic forms}\label{sec:ansatz}

In the rest of this section, we consider radially symmetric perturbation $f$ and weight $\rho$. Our goal is to construct the weight $\rho \leq e^{r^2}$ perturbed from $\mu^{-1}$ with 
\beq\label{eq:wg_ansatz}
\rho(r) = \mu^{-1} = \exp(r^2) , \ r \leq R_1, \quad \rho(r) \sim C r^k , \ r  \geq R_2, \ k > 2 
\eeq
for some large $R_1 < R_2$, such that 
\beq\label{eq:coer_ansatz}
\bal
  &- \la (\cL_{loc} + \cL_{nloc}) f, f \rho  \ra_{\R^3}  =
  Q_{\rho} + l.o.t. , \quad    F_{\rho, i} = \pa_i ( f \rho ), \\
& Q_{\rho} \teq
   \f{1}{2} \int_{\R^3} \int_{\R^3}  M_{\rho}(v, \td v) ( F_{\rho, i}(v) - F_{\rho,i}(\td v) )( F_{\rho,j}(v) - F_{\rho,j}(\td v)) \phi^{ij}(v - \td v) d v d \td v  ,
  \eal
\eeq
for some functions $ M_{\rho} \geq 0$ with $M_{\rho}(v, \td v) = M_{\rho}(\td v , v)$, where  l.o.t. is some lower order terms specified in \eqref{eq:lin_coer}.
Since $ \{ \phi^{ij} \}_{ij}$ \eqref{eq:nota1} is positive semi-definite and 
\[
 b_i  b_j \phi^{ij} \geq 0, \quad b_i = F_{\rho, i}(v) - F_{\rho,i}(\td v),
\]
the integrand in $Q_{\rho}$ is non-negative.  In the case of $\rho = \mu^{-1}$, \eqref{eq:coer_ansatz} is an exact identity with
\[
M_{\mu}(v, \td v) =  \mu(v) \mu(\td v), \quad 
F_{\mu, i} = \pa_i( f \mu^{-1} ), \quad l.o.t. = 0.
\]

We will only choose $\rho$ to modify $ \mu^{-1}$ in the far-field \eqref{eq:wg_ansatz} and expect that $M_{\rho}, F_{\rho,i}$ coincides with the above forms in the near-field. The diagonal part 
\beq\label{eq:coer_ansatz2}
\f{1}{2} \int_{\R^3} \int_{\R^3} M_{\rho}(v, \td v)( F_{\rho, i}(v) F_{\rho,j}(v) +F_{\rho, i}( \td v) F_{\rho,j}(\td v)  ) \phi^{ij}(v - \td v) 
\eeq
comes from the local part $\cL_{loc}$, and the remaining cross terms 
\beq\label{eq:coer_ansatz3}
- \f{1}{2}  \int_{\R^3} \int_{\R^3} M_{\rho}(v, \td v)( F_{\rho, i}(v) F_{\rho,j}(\td v) +F_{\rho, i}( \td v) F_{\rho,j}( v)  ) \phi^{ij}(v - \td v) 
\eeq
come from the nonlocal part $\cL_{nloc}$. The function $F_{\rho,i}$ and $M_{\rho}$ will be constructed in \eqref{eq:Jrho_comp0}, \eqref{eq:Jrho_comp3}, \eqref{eq:E1_loss_loc2}.

To establish \eqref{eq:coer_ansatz}, our benchmark is the coercive estimate $Q_{\mu^{-1}} \geq 0$.
We will perform several symmetrizations of the integrals and comparisons between the estimates for $Q_{\rho}$ and $Q_{\mu^{-1}}$. We modify the estimates in $\la \cL_{loc} f , f \rho \ra, \la \cL_{nloc} f, f \rho \ra$ in several steps with a small loss in each step, and eventually obtain the forms \eqref{eq:coer_ansatz2}, \eqref{eq:coer_ansatz3}, respectively. 


%

\subsection{Coercive estimate of the local part} 

Using \eqref{eq:lin_op2} and integration by parts, and then identities for radially symmetric functions $\pa_i \rho(v) = \f{v_i}{|v|} \rho_r$ \eqref{eq:Gauss_id1}, \eqref{eq:int_rad}, we get 
\beq\label{eq:lin_op3}
\bal
  - \int_{\R^3} \cL_{loc} f (v) \cdot f (v) \rho(v) dv  
= &  \int_{\R^3} \s^{ij} (\pa_i f +2 v_i f  )
\pa_j( f \rho) dv  \\
 = &    \int_{\R^3} \f{v_i v_j}{|v|^2} \s^{ij}
 (\pa_r f + 2 r f ) (\pa_r f \rho + f \pa_r \rho) d v  = 4 \pi J_{\rho}(f, f), \\
 \quad 
 J_{\rho}(f, f) \teq &  \int_0^{\infty} (- \bar g_{rr} r^2) (\pa_r f + 2 r f ) (\pa_r f \rho + f \pa_r \rho) d r .
 \eal 
\eeq

Firstly, we establish that $J_{\rho}$ is non-negative up to a small loss in the far-field 
by carefully designing $\rho$ with the properties\eqref{eq:wg_ansatz}. Then we will show that $J_{\rho}$ is close to the form \eqref{eq:coer_ansatz2} with a small loss. 
If $\rho = e^{r^2}$, since $ \pa_r\rho=  2 r \rho$ and $-\bar g_{rr} \geq 0$ \eqref{eq:Gauss_id1},  we obtain that $ J_{\rho}(f,f) \geq 0$ and $-\cL_{loc}$ is coercive. We drop the dependence of $J_{\rho}(f,f)$ on $f$ when there is no confusion. We will use the relation \eqref{eq:int_rad} between integrals over $\R_+$ and $\R^3$ for radially symmetric functions frequently.


Using \eqref{eq:lin_op3}, $\pa_r \rho = r q $ \eqref{eq:nota2}, and integration by parts, we derive 
\beq\label{eq:lin_op_Jloc1}
J_{\rho} = \int_0^{\infty} (-r^2 \bar g_{rr}) (\pa_r f)^2 \rho dr 
 + \int_0^{\infty} f^2 \B( ( r^3 \bar g_{rr} \rho)_r 
+ \f{ (\bar g_{rr} r^3 q )_r}{2} - r^2 \bar g_{rr} \cdot 2 r^2 q  \B) dr  \teq I + II.
\eeq
The first term $I$ is non-negative. 

Using \eqref{eq:Gauss_id1} and $r q = \rho_r$  \eqref{eq:nota2}, we reformulate the coefficient in $II$ as follows 
\[
\bal
( r^3 \bar g_{rr} \rho )_r 
+ \f{ (\bar g_{rr} r^3 q )_r}{2} - r^2 \bar g_{rr} \cdot 2 r^2 q 
&= (r^3 \bar g_{rr})_r ( \rho + \f{q}{2}) 
+ (-\bar g_{rr} r^3) ( 2 r q - \f{q^{\prime}}{2}- \rho_r) \\
& =   - \f{1}{2}  \bar f r^2 \cdot ( \rho + \f{q}{2})
+ (-\bar g_{rr} r^3) (  r q - \f{q^{\prime}}{2}) \\
\eal
\]

We want to compare $J_{\rho}$ with the non-negative form 
\beq\label{eq:lin_op_Jloc2}
\td J_{\rho} 
\teq \int_0^{\infty} (-\bar g_{rr} r^2) (\pa_r f + \lam f)^2 \rho d r
\eeq
for some $\lam$ depending on $\rho$ to be determined. See \eqref{eq:lin_op_lam}. Expanding the square and then performing integration by parts, we can rewrite $\td J_{\rho}$ as follows
\[
\bal
 \td J_{\rho} 
& = \int (-\bar g_{rr} r^2) (\pa_r f)^2 \rho  +
(-\bar g_{rr} r^2  \rho ) \cdot ( 2 \pa_r f \cdot  \lam f + \lam^2 f^2 ) dr \\
& = \int 
(-\bar g_{rr} r^2) (\pa_r f)^2 \rho dr 
+\int  (\bar g_{rr} r^3 \cdot \f{\lam \rho}{r} )_r f^2 
+  \lam^2 (-\bar g_{rr} r^2) f^2 \rho dr \teq \td I + \wt{II}.
\eal
\]

Using \eqref{eq:Gauss_id1}, we can simplify the coefficient of $f^2$ in $\wt{II}$ as follows 
\[
(\bar g_{rr} r^3)_r \cdot  \f{\lam \rho}{r}
 + (\bar g_{rr} r^3) \cdot ( \f{\lam \rho}{r} )_r 
 + \lam^2 (-\bar g_{rr} r^2 \rho)
 = - \f{1}{2} r^2 \bar f \cdot \f{\lam \rho}{r} 
 + (- \bar g_{rr} r^3) \cdot \B(  -( \f{\lam \rho}{r} )_r  + \f{\lam^2 \rho}{r}  \B).
\]

Note that $\td I = I$. We yield 
\beq\label{eq:J_comp}
J_{\rho} - \td J_{\rho}
= II - \wt {II} = 
\int (-\f{1}{2}\bar f r^2)( \rho + \f{q}{2} - \f{\lam \rho}{r} ) f^2
+ (-\bar g_{rr} r^3) \B\{ r q - \f{q^{\prime} }{2} -  \B( -( \f{\lam \rho}{r} )_r  + \f{\lam^2 \rho}{r}  \B)   
\B\} f^2 .
\eeq

We will choose $\rho =\exp(r^2) $ for $r \leq R_1$ with $R_1$ sufficiently large and $ \lam  = \f{qr}{\rho}$ \eqref{eq:lin_op_lam}. Then the integrand in the first part in $J_{\rho} - \td J_{\rho} $ is $0$ for $r \leq R_1$ since $\rho + \f{q}{2} - \f{\lam \rho}{r} = 0$ \eqref{eq:nota2}. 
Beyond $R_1$,  
the first term reduces to 
\beq\label{eq:J_comp1}
 \int - \f{1}{2} \bar f r^2 (\rho + \f{q}{2}- q ) f^2 
 =  \int_{R_1}^{\infty}  \f{1}{2} \bar f r^2 (  \f{q}{2} - \rho) f^2  , 
\eeq
which is very small since $\bar f$ decays exponentially fast.  We defer the estimate to \eqref{eq:E1_loss2}.

\subsubsection{Pointwise comparison and design of $\rho$}

For the second term in $J_{\rho} - \td J_{\rho}$, we want to obtain that it is non-negative up to a similar small term by establishing 
\beq\label{eq:lin_op_ineq1}
r q - \f{q^{\prime}}{2} \geq 
  -( \f{\lam \rho}{r} )_r  + \f{\lam^2 \rho}{r} 
  + l.o.t.,
\eeq
where $l.o.t.$  denotes lower order terms, which grows at most polynomially. 

\begin{remark}
Below, we will first derive $\lam, q, \rho$ and then specify the lower order terms. 
Most inequalities below are equality for $\rho = \exp(|v|^2)$ with \eqref{eq:nota2}
and most $l.o.t.$ terms are $0$ for $\rho = \exp(|v|^2)$. 
We perturb $\rho$ from $\mu^{-1}$ carefully to achieve \eqref{eq:wg_ansatz}.
\end{remark}

We define 
\beq\label{eq:lin_op_lam}
\lam \teq \f{ \rho_r}{\rho} = \f{q r}{\rho}, 
\quad \lam_i \teq \f{ \pa_i \rho}{\rho} = \f{v_i}{|v|} \f{ \rho_r}{\rho} = \f{v_i}{|v|} \lam.
\eeq
Using \eqref{eq:nota2}, we get 
\[
 \f{\lam \rho}{r} = \f{\rho_r}{r} = q , \quad \f{\lam^2 \rho}{r} = \f{q^2 r}{\rho} .
\]

With the above choice of $\lam$, we can reduce \eqref{eq:lin_op_ineq1} to
\[
 r q + \f{q^{\prime}}{2} \geq \f{q^2 r }{\rho} + l.o.t.
\]

We will choose $q$ such that
\beq\label{eq:lin_op_con1}
0 \leq q^{\prime }  = 2 r q \eta
\eeq
where we use \eqref{eq:nota2}. Due to the asymptotics \eqref{eq:wg_ansatz} of $\rho$ in the far-field, we yield 
\beq\label{eq:eta_ineq1}
q \sim C r^{k-2}, \quad \eta = \f{k-2}{2} r^{-2}, \quad k > 2,
\eeq
for $r > R_1 $ large enough.

Since $\rho,  r q + \f{q^{\prime}}{2} > 0 $, to get the inequality above \eqref{eq:lin_op_con1},
we further consider 
\[
\rho \geq \f{q^2 r }{ q r +\f{q^{\prime} }{2} } + l.o.t.
= \f{ q^2 r}{ qr(1 + \eta)} + l.o.t.
 = \f{ q }{ 1 + \eta } + l.o.t.
\]

We will choose $\rho = e^{r^2}, q = 2 \rho, \eta = 1$ such that $\rho = \f{q}{1 + \eta}$ for $r \leq R_1$. Thus, to obtain the above inequality, we further consider 
\[
 \rho^{\prime} \geq \f{d}{dr} (  \f{ q }{ 1 + \eta } ) + l.o.t.
\]

From $ \rho^{\prime} = rq $ \eqref{eq:nota2}, we further consider
\beq\label{eq:undo3}
\bal
 r q  & \geq \f{ q^{\prime}(1 + \eta) -  q \eta^{\prime}}{ (1 +\eta)^2} + l.o.t.
= \f{ 2 r q \eta (1 + \eta) -  q \eta^{\prime}}{ (1 +\eta)^2} + l.o.t. \\
 r(1 + \eta)^2 & \geq 2 r \eta(1 + \eta) - \eta^{\prime} + l.o.t., \quad r(1 - \eta^2)  \geq - \eta^{\prime} + l.o.t.
 \eal
\eeq

The last constraint is an ODE inequality on $\eta$ only. In view of \eqref{eq:lin_op_con1}, we choose $\eta \geq 0$. We want to construct $\rho \leq \exp(r^2)$. From \eqref{eq:nota2} and its derivation below, we want $q \leq 2 e^{r^2}, \eta \leq 1$ for all $r \geq 1$. Thus, we focus on the above constraint on $\eta$ with $\eta \in [0, 1]$. 

Since $\eta \leq 1$, from \eqref{eq:nota2} and $q- 2 \rho =0, r \leq R_1$ \eqref{eq:wg_ansatz}, we yield 
\[
(q - 2 \rho)^{\prime} = q^{\prime} - 2 \rho^{\prime} 
= 2  r q \eta - 2 r q \leq 0, \quad q \leq 2 \rho. 
\]

To solve \eqref{eq:undo3}, we consider 
\beq\label{eq:undo4}
\bal
\eta^{\prime}  &\geq  r(\eta - 1) - e^{- 2 r^2} \geq r(\eta^2 - 1) - e^{- 2 r^2}.
\eal
\eeq
The first ODE is equivalent to 
\[
 \f{d}{dr} \B(  (\eta-1) e^{-r^2/2} \B)  \geq  - e^{-5 r^2/2 } .
\]	
In particular, we first construct $\eta_0$
\[
\bal
& (\eta_0 - 1 ) e^{-\min(r, R_2)^2/2} = - \one_{r \geq R_1} \int_{R_1}^{ \min(r, R_2)} e^{-5 s^2/2} ds,  \\
& \eta_0 = 1 -\one_{r \geq R_1}  \int_{ R_1}^{ \min(r, R_2)} e^{\min(r, R_2)^2/2 - 5 s^2/2} dr
\eal
\]
with $R_2$ satisfying 
\[
F(R_2) = 1, \quad F(r) \teq e^{r^2 / 2} \int_{R_1}^{r}  e^{-5 s^2/2} ds  
\]
Note that $F(R_1) = 0, F(\infty)=\infty$ and $F$ is increasing, there is a unique $R_2$ such that $F(R_2) = 1$. Moreover, since 
\[
F(  \sqrt{5} (R_1+1) ) > e^{ 5(R_1 +1)^2/2} \int_{R_1}^{R_1 + 1} e^{-5 s^2/2} ds > 1,
\]
we get $R_2 \leq \sqrt{5} (R_1 + 1)$. The above construction satisfies 
\[
\eta_0(r) = \one_{r \leq R_1} + \one_{r \in [R_1, R_2]} (1 - F(r)) .
\]

For $r \in [R_1, R_2]$, $\eta_0$ satisfies the first ODE in \eqref{eq:undo4}. Thus, we have 
\beq\label{eq:eta0_ineq1}
\eta_0^{\prime} \geq  r(\eta_0-1) - e^{- 2 r^2} \one_{r \in [R_1, R_2]}
 \geq r(\eta_0^2-1) - \cE, 
\quad \cE \teq   e^{- 2 r^2} \one_{r \in [R_1, R_2]}. 
\eeq

We will treat $\cE$ as the lower order term appeared in, e.g. \eqref{eq:undo3}.

Next, we construct $\eta$ based on $\eta_0$. Let $R_1^* \in [R_1, R_2]$ be the smallest solution of 
\[
\eta_0(r) = \f{k-2}{ 2 r^2 } .
\]

We choose $k \leq 20, R_1 > 10$ so that $ \f{k-2}{ 2 R_1^2 } < \f{1}{2}$. Since $d(r) = \eta_0(r) - \f{k-2}{ 2 r^2 } $ satisfies $d(R_1)>0 , d(R_2) < 0$, $R_1^*$ exists. To further impose the decay condition \eqref{eq:eta_ineq1}, we construct 
\beq\label{eq:eta0_ineq12}
 \eta(r) = \eta_0(r), \ r \in [R_1, R_1^*], \quad \eta(r) =  \f{k-2}{ 2 r^2 } , \quad r \geq R_1^*.
\eeq
Since $R_1^*$ is the smallest solution, we yield 
\beq\label{eq:eta0_ineq2}
\eta(r) = \eta_0(r) \geq \f{k-2}{ 2 r^2 } , \quad  r \in [R_1, R_1^*], 
\quad \eta(r) = \f{k-2}{ 2 r^2 } , \quad r \geq R_1^*.
\eeq

By choosing $R_1$ large enough, e.g. $R_1>10$, and restricting $k<20$, we get
\[
\eta(R_1^*) = \eta_0(R_1^*) = \f{k-2}{2 (R_1^{*})^2 } \leq \f{k-2}{2 R_1^2} < \f{1}{2} .
\]
As a result, for $r \geq R_1^* > 10$, we get $\eta(r) = \f{k-2}{2 r^2} < \f{1}{2}$,
\[
r(1 - \eta) + \eta^{\prime} \geq r \cdot \f{1}{2}
- \f{k-2}{r^3}
\geq r \cdot \f{1}{2} - \f{1}{2} \cdot \f{ 2}{r} > \f{r}{4}.	
\]

Hence, \eqref{eq:eta0_ineq1} still holds for $\eta$. Below, we specify all the lower order terms.


Using the above functions, \eqref{eq:eta0_ineq1},  and undoing the previous derivations \eqref{eq:undo3}, we yield 
\[
 \f{d}{dr}( \rho - \f{q}{1 + \eta} ) = rq - \f{2 r q \eta(1 + \eta) - q \eta^{\prime} }{(1 + \eta)^2}
 \geq r q \B( 1 - \f{2 \eta(1 + \eta)}{(1+\eta)^2}  \B)
 + \f{q r(\eta^2-1)}{ (1+\eta)^2} - \f{q \cE}{(1+\eta)^2}.
\]

The terms except for $\cE$ are canceled 
\[
1 - \f{2 \eta(1+\eta)}{(1+\eta)^2} + \f{(\eta^2-1)}{ (1+\eta)^2}
= \f{1 -\eta^2}{ (1 + \eta)^2 } + \f{\eta^2-1}{ (1 + \eta)^2 } = 0. 
\]

It follows 
\[
 \f{d}{dr}( \rho - \f{q}{1+\eta}) \geq - q \cE . 
\]

Since $\eta \leq 1$, solving the ODE \eqref{eq:nota2} with $q(r) = 2 e^{r^2}, r \leq R_1$, we get $q \leq 2 e^{r^2}$. 
Since $ \rho - \f{q}{1 + \eta} = 0$ at $r=0$ \eqref{eq:nota2}, \eqref{eq:wg_ansatz}, integrating $r$, we derive 
\[
\rho - \f{q}{1 + \eta} \geq - 2  \one_{r \geq R_1} \int_{R_1}^{\min(r, R_2)} e^{-s^2} ds 
\geq -  2\one_{r \geq R_1} (R_2 - R_1) e^{- R_1^2}
\teq -  \one_{r \geq R_1}  \e_2 ,
\]
where 
\beq\label{eq:E1_loss1_e2}
\e_2 \teq 2 (R_2 - R_1) e^{-R_1^2}. 
\eeq

Since $R_2 \leq \sqrt{5}( R_1 + 1)$, by choosing $R_1 \geq 1$ large enough, we can make $\e_2 $ very small. Using the above estimate and $ \f{q}{1+\eta} -\e_2 \geq \f{1}{2}-\f{1}{4} >0$ ,we derive 
\[
\bal
 rq + \f{ q^{\prime}}{2} - \f{q^2 r}{\rho}
= & rq (1 + \eta - \f{q}{\rho} )  \geq   r q \B(1 + \eta - \f{q}{ \f{q}{1 + \eta} - \one_{r \geq R_1 } \e_2  } \B) \\
= & r q \B( 1 + \eta - \f{q (1 + \eta)}{q - (1 + \eta) \one_{r \geq R_1} \e_2}   \B)
= - r q \cdot \f{ (1 + \eta)^2 \one_{r \geq R_1} \e_2}{ q - (1 + \eta) \one_{r \geq R_2} \e_2 } .
\eal
\]
Since $\e_2 \leq \f{1}{4} q, \eta \in [0, 1]$, we yield 
\beq\label{eq:E1_loss1}
rq + \f{ q^{\prime}}{2} - \f{q^2 r}{\rho} \geq - 4 r q \one_{r \geq R_1} \e_2  \cdot \f{1}{q/2}
 = - 8 r \one_{r \geq R_1} \e_2. 
\eeq
The right hand side is a small loss growing polynomially. We achieve \eqref{eq:lin_op_ineq1}.

\subsubsection{Estimate of the first term}

The estimate of \eqref{eq:J_comp1} is simple. Using $\eta \leq 1$, $(2 \rho - q)_r = 2 q r(1 - \eta) \geq 0$, and $\rho(R_1) = e^{R_1^2} = q(R_1)/2$, we yield 
\[
 0 \leq 2 \rho - q  \leq \one_{r \geq  R_1} 2 \rho 
 \leq \one_{r \geq R_1} 2 e^{r^2}. 
\]
It follows 
\beq\label{eq:E1_loss2}
 |  \int  \bar f r^2 (  \f{q}{2} - \rho) f^2 | 
 = \int \bar f r^2 (\rho - \f{q}{2}) f^2 d r 
 \leq  \int \one_{r \geq R_1} f^2 r^2 d r.
\eeq

Recall from \eqref{eq:lin_op_lam} that 
\[
r q - \f{q^{\prime}}{2} - \B(  -( \f{\lam \rho}{r} )_r  + \f{\lam^2 \rho}{r}  \B)
  =  rq + \f{ q^{\prime}}{2} - \f{q^2 r}{\rho} . 
\]

Plugging the estimates \eqref{eq:E1_loss1}, \eqref{eq:E1_loss2} in \eqref{eq:J_comp} and using the definition of $\td J_{\rho}$ \eqref{eq:lin_op_Jloc2}, we establish 
\beq\label{eq:E1_loss_loc1}
\bal
- \f{1}{4 \pi} \la \cL_{loc} f ,  f \rho \ra 
= J_{\rho} 
& \geq \td J_{\rho}
- \int 8 r (-\bar g_{rr} r^3) f^2 \one_{r \geq R_1} \e_2
- 2 \int_{r \geq R_1} f^2 r^2 d r \\
& \geq \td J_{\rho} - C  \int \one_{r \geq R_1} f^2 r^2 d r.
\eal
\eeq
for $R_1$ large enough, where $C$ is some absolute constant and we have used 
\[
|\bar g_{rr}| \les (1 + r)^{-3}, \quad \e_2 \les 1.
\]


\subsubsection{Asymptotics of $q, \rho$}

Since $\eta = \f{k-2}{2 r^2}, r \geq R_2$ in the far-field, from \eqref{eq:nota2}, we yield 
\[
q^{\prime } = (k-2) r^{-1} q , \quad q(r) = q(R_2),  \quad \rho_r = q r, \quad \rho \asymp 1 + r^k 
\]
for $r \geq R_2$. From the definition of $\rho, q, \eta$ \eqref{eq:nota2}, we yield 
\[
\bal
q(r)  &= q(R_1) \exp( \int_{R_1}^r  2 s \eta(s) ds ), \quad q(R_1) = 2 \mu^{-1}(R_1)  , \\
\rho(r)  & = \rho(R_1) + \int_{R_1}^r  s q(s) ds 
= \mu^{-1}(R_1) \B( 1 +   \int_{R_1}^r 2 s \exp(  \int_{R_1}^s 2 \tau \eta(\tau) d \tau  ) ds  \B) .
\eal
\]

Due to \eqref{eq:eta0_ineq2}, we yield 
\beq\label{eq:rho_asym1}
\bal
\rho(r) 
& \geq \mu^{-1}(R_1) \B(1 + \int_{R_1}^r 2 s \exp( \int_{R_1}^s  \f{k-2}{\tau}   d \tau  ) \B) \\
 & = \mu^{-1}(R_1) \B(1 + \int_{R_1}^r 2 s ( \f{s}{R_1})^{k-2 } d s \B) 
  = \mu^{-1}(R_1)( 1 + 2 \cdot \f{r^k - R_1^k}{ k R_1^{k-2}} ) .
 \eal
\eeq

Using $\eta = \f{k-2}{2 r^2} + \one_{r \leq R_2}$ \eqref{eq:eta0_ineq2}, 
 for $ r > 10 R_1 >  R_2 $, we yields 
\beq\label{eq:rho_asym3}
\bal
 \int_{R_1}^r 2 \tau \eta(\tau) d \tau 
 & \leq  \min(C R_1^2, r^2- R^2_1) + \int_{R_1}^r 2 \tau \cdot \f{k-2}{2 \tau^2} d \tau  \leq C R_1^2 + (k-2) \log( r / R_1) , \\
  & \exp(  \int_{R_1}^r 2  \tau \eta(\tau) d \tau  )
   \leq \exp( C R_1^2) (r / R_1)^{k-2},  \\
 q(r) &  \leq q(R_1) \exp( C R_1^2) (r / R_1)^{k-2},  \\
\rho & \leq q(R_1)( 1 + \int_{R_1}^r  s q(s) ds)
\leq C q(R_1) \exp(C R_1^2 ) (1 + \int_{R_1}^r \f{ s^{k-1}}{ R_1^{k-2}} ds)  \\
&\leq C q(R_1) \exp(C R_1^2 ) (1 +  \f{r^k - R_1^k}{k R_1^{k-2}} )
\leq C_{\e} q(R_1) \exp( C R_1^2) \exp( \e (r^2 - R_1^2)),
\eal
\eeq
where we have used $r > 10 R_1 > 100$ in the last inequality.

Note that $\mu(r) = \exp(-r^2), q(R_1) = 2 \mu(R_1)^{-1}$, therefore, we yield 
\[
\bal
\rho \mu & \leq C_{\e} \exp( C R_1^2) 2 \mu(R_1)^{-1}
\exp( \e(r^2 - R_1^2) )\cdot \mu(R_1) \exp( - ( r^2 - R_1^2))   \\
& \leq C_{\e} \exp( C R_1^2 ) \exp( -(1 -\e) (r^2 - R_1^2) ) .
\eal
\]

In particular, for $r \geq m_{\e} R_1$ with some constant $m_{\e} > 10$, we yield 
\beq\label{eq:rho_asym2}
\rho \mu \leq C_{\e} \exp( - (1 - 2 \e) (r^2 - R_1^2)).
\eeq


\subsubsection{Symmetrization of the integral in $\R^3$}

In the previous section, we have constructed $\rho$ and obtained that $J_{\rho}$ is non-negative up to a small loss \eqref{eq:E1_loss_loc1}. Below, we further modify $J_{\rho}, \td J_{\rho}$ to obtain \eqref{eq:coer_ansatz2} with a small loss. Denote 
\beq\label{eq:Jrho_comp0}
\bal
& F_i = \pa_i f + \f{\pa_i \rho}{\rho} f
= \pa_i f + \lam_i f  , \quad F_{\rho, i} =\pa_i( \rho f)
= \rho \pa_i f + \pa_i \rho \cdot f = \rho F_i 
,  \\
&  M(v, \td v) = \f{1}{2} \f{  (\mu \rho) (\td v) + (\mu \rho) (v) }{ \rho(v) \rho(\td v)},
\quad D(v, \td v) = \f{1}{2} \f{  (\mu \rho) (\td v) - (\mu \rho) (v) }{ \rho(v) \rho(\td v)} .
\eal
\eeq
We choose the above $F_{\rho, i} $ in \eqref{eq:coer_ansatz}. Since $f, \rho$ are radial, 
using \eqref{eq:lin_op_lam}, we get
\beq\label{eq:Jrho_comp02}
F_i = \f{v_i}{|v|} ( \pa_r f + \f{\pa_r \rho }{\rho} f ) = 
\f{v_i}{|v|}  F_0, \quad F_0 \teq \pa_r f + \f{\pa_r \rho }{\rho} f = \pa_r f + \lam f.
\eeq

From \eqref{eq:nota1}, we have 
\[
\bar g_{rr} = \f{v_i v_j}{|v|^2}  \pa_i \pa_j \bar g(v) 
 =  \f{v_i v_j}{|v|^2}  \pa_i \pa_j (-\D)^{-2} \mu(v) 
 = -\f{v_i v_j}{|v|^2}  \phi^{ij} \ast \mu,
  \quad \bar g_{rr} F_0^2 = - \phi^{ij} \ast \mu(v) \cdot  F_i(v) F_j(v) ,
\]
where we sum the repeated index over $i=1,2,3$. We expand the convolution in $-\bar g_{rr}$ in $\td J_{\rho}$ \eqref{eq:lin_op_Jloc2} and use \eqref{eq:int_rad} to rewrite it as follows 
\beq\label{eq:Jrho_comp1}
\bal
\td J_{\rho} 
& \teq \int_0^{\infty} (-\bar g_{rr} r^2) (\pa_r f + \lam f)^2 \rho d r 
\\
& = \f{1}{4\pi}  \int_{\R^3}  (\phi^{ij} \ast \mu ) F_i(v) F_j(v) \rho d v
 = \f{1}{4\pi} \int_{\R^3 \times \R^3}  \phi^{ij}(v - \td v) \f{ \mu( \td v )}{\rho(v)} (\rho F_i)(v) ( \rho F_j(v))  d \td v d v.
 \eal
\eeq
Below, we drop the domain $\R^3$ for $v, \td v$ for simplicity. To connect the above form and \eqref{eq:coer_ansatz2}, we write 
\[
 \f{ \mu( \td v )}{\rho(v)} = \f{ (\mu \rho)(\td v) }{\rho(v)  \rho(\td v)} 
  =  M(v, \td v) + D(v, \td v),
\]
where $M, D$ are defined in \eqref{eq:Jrho_comp0}. We want to compare $\td J_{\rho}$ with 
$\f{1}{4\pi}$ of \eqref{eq:coer_ansatz2}
\[
\bal
\td { \td {J_{\rho} }} 
&\teq  \f{1}{4\pi} \int  \phi^{ij}(v - \td v)  M(v, \td v) (\rho F_i)(v) ( \rho F_j(v))  d \td v d v \\
&= \f{1}{8 \pi} \int  \phi^{ij}(v - \td v)  M(v, \td v) ( F_{\rho, i}(v) F_{\rho, j}(v)
+  F_{\rho, i}(\td v) F_{\rho, j}(\td v) )  d \td v d v ,
\eal
\]
where we have symmetrized the integrand in the last equality, and $\phi^{ij}(v - \td v), M(v, \td v)$ are symmetric in $v, \td v$. The factor $4\pi$ comes from the relation between radial integral and integral in $\R^3$ \eqref{eq:int_rad}. From \eqref{eq:deri}, for a radially symmetric function $A$, we yield 
\beq\label{eq:Gauss_id2}
  \f{v_i}{|v|} \f{v_j}{|v|} \int \phi^{ij}(v - \td v) A(\td v )  d \td v
  =  -  \f{v_i}{|v|} \f{v_j}{|v|}   \pa_{ij}(-\D)^{-2} A 
  = - \pa_{rr}(-\D)^{-2} A .
\eeq
We adopt the notation above \eqref{eq:deri} and denote 
\[
g_A \teq (-\D)^{-2} A, \quad \bar g = (-\D)^{-2} \mu.
\]

Simplifying the integral over $\td v$ in $\td { \td {J_{\rho} }}$, and rewriting it as an 1D integral following \eqref{eq:Jrho_comp1}, we obtain
\[
\bal
& \td { \td {J_{\rho} }} 
= \int_0^{\infty} \f{1}{2} (-r^2) (  g_{\rho^{-1}, rr}  \mu(v) 
+ \bar g_{rr}  \rho^{-1}(v) ) (\pa_r f + \f{\pa_r \rho}{\rho} f  )^2 \rho^2  d r , \\
 \eal
\]

To compare $\td J_{\rho}$ and $\td {\td { J_{\rho}}}$, we only need to compare
\beq\label{eq:Jrho_comp2}
  N \teq - \bar g_{rr} \rho^{-1}, \quad  \td N = -  g_{ \rho^{-1}, rr} \mu. 
\eeq
pointwisely 
\beq\label{eq:gam_ineq0}
N(r) \geq \td N(r) (1 - \d(r)) > 0
\eeq
for some small $\d$. Using the formula of $g_{rr}$ \eqref{eq:BSlaw3}, we yield 
\[
\bal
N = \f{1}{3} \B( \f{1}{r^3} \int_0^r \mu s^4 + \int_r^{\infty} \mu s \B) \rho^{-1}, \quad
\td N =\f{1}{3} \B( \f{1}{r^3} \int_0^r \rho^{-1} s^4 + \int_r^{\infty} \rho^{-1} s \B) \mu .
\eal 
\]

We observe that for large $r$, $\mu$ decays much faster than $\rho^{-1}$ \eqref{eq:wg_ansatz}, which allows us to prove the pointwise bound. Next, we estimate $\g$ defined below 
\beq\label{eq:gam_ineq1}
\g \teq \f{ N - \td N }{ N }, 
\quad \g \geq  - \d_{\e}(r) , 
\quad \d_{\e}(r) \teq  
 C_{k} e^{-R_1^2} R_1^5 \one_{ r \leq m_{\e} R_1} 
+ C_{k, \e} e^{- (1-\e) r^2} \one_{ r > m_{\e} R_1},
\eeq
for any $\e \in (0, 1/4)$ with some $C_{\e}, m_{\e} > 1$. Later, we will fix $k=5/2$ in \eqref{eq:wg_para} and then $\d_{\e}$ is independent of $k$. 
Since $\rho_r =r q, \mu_r =  -2 r \mu$ \eqref{eq:nota2}, 
\[
(\rho \mu)^{\prime} = \rho^{\prime} \mu + \rho \mu^{\prime}
 = \mu ( r q  - 2 r \rho  ) = \mu r (q - 2 \rho ) , \quad 
 (q - 2\rho)^{\prime} = 2 \eta r q - 2 r q  \leq 0, 
\]
and $q = 2 \rho = 2 \mu^{-1}$ for $r \leq R_1$, we yield 
\beq\label{eq:pq_size}
q \leq 2 \rho , \quad  \rho \mu \leq 1 ,\quad \rho^{-1} \geq \mu ,
\eeq
and $\rho \mu$ is decreasing for all $r \geq 0$. It follows 
\[
 \mu(s) \rho^{-1}(r) \geq \rho^{-1}(s) \mu(r), \quad s \leq r.
\]
Using 
\[
N - \td N \geq \f{1}{3} \int_{\max(r, R_1)}^{\infty} (\mu - \rho^{-1})(s) ds \cdot \mu(r), \quad 
  N  \geq C  \min(r^{-3}, 1 ) \rho^{-1} ,
\]
and $\mu - \rho^{-1} \leq 0$, we yield 
\[
 \g = \f{ N - \td N}{N  } 
 \geq \f{1}{3N} \int_{\max(r, R_1)}^{\infty} (\mu - \rho^{-1})(s) ds \cdot \mu(r)
 \geq C \max(r^3, 1) \int_{ \max(r, R_1)}^{\infty}( \mu - \rho^{-1}) s  ds  \cdot  (\rho \mu)(r). 
\]

For any $m > 1, 2<k \leq 20$ and  $r \leq m R_1 $, using \eqref{eq:rho_asym1} and 
\[
\rho \geq \mu^{-1}(R_1) , \ r \in [R_1, 2 R_1],  \quad \rho \geq C \mu^{-1}(R_1) (1 + r^k) R_1^{2-k},  
\ r \geq 2R_1,  \quad \mu \rho \leq 1,
\]
 we yield 
\[
\int_{\max(r, R_1)}^{\infty} |\mu - \rho^{-1}| s ds 
  \leq \int_{\max(r, R_1)}^{\infty} \rho^{-1} s ds
 \leq C \mu(R_1) 
  \B( \int_{R_1}^{2 R_1} s ds+ R_1^{k-2} \int_{2 R_1}^{\infty} s^{1- k } ds \B) \les_k \mu(R_1) R_1^2,
  \quad 
\]

It follows 
\[
\g \geq - C_k \max(r^3, 1)  R_1^2 \mu(R_1).
\]

For any $\e \in (0, \f{1}{4})$, using \eqref{eq:rho_asym2} with $m = m_{\e}$, for $ r > m_{\e} R_1> R_1 > 10$, we yield 
\[
\bal
 C \max(r^3 , 1) \int_{ r}^{\infty} \rho^{-1} s ds (\rho \mu)
 &\leq C_{k,\e} r^3 R_1^2 \exp( - (1-2 \e)( r^2 - R_1^2)  ) \\
 & \leq C_{k,\e} \exp( - (1-3\e) r^2 ),
 \eal
\]
where $C_{k,\e}$ is some absolute constant depending on $k, \e$. It follows 
\[
\g \geq - C_{k,\e} \exp( - (1-\e)(r^2 - R_1^2) ) 
\]
for $r \geq  m_{\e} R_1$, where we have renamed the variable $\e$. Combining the above two estimates, we establish \eqref{eq:gam_ineq1}. 
Using \eqref{eq:gam_ineq1}, \eqref{eq:gam_ineq0}, we yield the pointwise estimate
\[
N(r) - \td N(r) \geq - \d_{\e}(r) N(r),
\quad 
N(r) \geq \f{1}{2} (N(r) + \td N(r)) (1  - \d_{\e}(r)). 
\] 
Using derivations similar to \eqref{eq:Jrho_comp1} and the definitions
\eqref{eq:Jrho_comp0} of $M$ and \eqref{eq:Jrho_comp2},  we get 
\[
\bal
 {\td{J_{\rho} }}
&\geq \f{1}{2} \int r^2 ( -  \bar g_{rr} \rho^{-1} 
- g_{\rho^{-1}, rr} \mu   )(1 - \d_{\e}(r))  (\pa_r f + \lam f)^2 \rho^2  d r  \\
& = \f{1}{4\pi}  \int \phi^{ij}(v - \td v) M(v, \td v) (1 - \d_{\e}(|v|)) 
 \rho F_i(v)  \rho F_j(v) d v d \td v .
 \eal
\]
The factor $\f{1}{4\pi}$ comes of the relation  \eqref{eq:int_rad}.
Since $\{ \phi^{ij}\}_{ij}$ is positive semi-definite, we yield 
\[
\phi^{ij}(v - \td v)  \rho F_i(v)  \rho F_j(v) \geq 0,
\]
which along with $1 - \d_{\e}(|v|) \geq 1 - \d_{\e}(|v|) - \d_{\e}( |\td v|), M(v, \td v) \geq 0$  implies 
\beq\label{eq:Jrho_comp3}
\bal
& \td J_{\rho}
\geq J_{\e, \rho} ,  \quad  
 M_{\e} \teq M(v, \td v)(1 - \d_{\e}(|v|) - \d_{\e}( |\td v|)) ,\\
 & J_{\e, \rho } \teq  \f{1}{4\pi} \int  \phi^{ij}(v - \td v)  M_{\e}(v, \td v)  F_{\rho, i}(v) F_{\rho, j}(v)
d \td v d v .    \\
\eal
\eeq
Since $M_{\e}(v, \td v), \phi^{ij}(v - \td v)$ are symmetric in $v, \td v$, symmetrizing the integrand in $J_{\e, \rho}$, and combining \eqref{eq:Jrho_comp3}, \eqref{eq:E1_loss_loc1}, we establish
\beq\label{eq:E1_loss_loc2}
\bal
 & - \int_{\R^3} \cL_{loc} f \cdot f \rho d v 
 \geq 4 \pi \td J_{ \rho}  - C \int_{r \geq R_1} f^2 r^2 d r   \geq 4 \pi J_{\e, \rho}  - C \int_{r \geq R_1} f^2 r^2 d r   \\
 & \geq \f{1}{2}\int  \phi^{ij}(v - \td v)  M_{\e}(v, \td v)  ( F_{\rho, i}(v) F_{\rho, j}(v)
+  F_{\rho, i}(\td v) F_{\rho, j}( \td v) ) 
d \td v d v  - C \int_{r \geq R_1} f^2 r^2 d r  .
\eal
\eeq
We establish the relation between \eqref{eq:coer_ansatz2} and \eqref{eq:lin_op3} up to a small loss. We will further estimate the error term in Section \ref{sec:E1_lower} using the energy.

\subsection{Estimate of the nonlocal parts}

In this section, we estimate $\cL_{nloc}$ and show that it is close to \eqref{eq:coer_ansatz3} with a small error.

Recall the definitions of $\lam, \lam_i$ from \eqref{eq:lin_op_lam} and $F_{\rho,i}$ from 
\eqref{eq:Jrho_comp0}.  Using the definition \eqref{eq:lin_op2} and integration by parts, we yield 
\beq\label{eq:nloc_form1}
\bal
  I_{nloc}  & \teq  \int - \cL_{nloc} f \cdot f \rho d v = \int_{\R^3} f \rho \cdot \pa_i ( \mu \phi^{ij} \ast (\pa_j f + 2 v_j f ) ) d v   \\
 & = -\int_{\R^3} \mu(v)  \phi^{ij}(v - \td v) \pa_i ( f \rho )(v) 
 \cdot  ( \pa_j f(\td v) + 2 \td v_j f(\td v) ) d \td v d v \\
 \eal
\eeq
Recall $ \pa_i( f \rho) = F_{\rho,i}= \rho F_i$ \eqref{eq:Jrho_comp0}. 
We further rewrite $I_{nloc}$ as follows 
\[
\bal
I_{nloc} & =
- \int_{\R^3} \mu(v) 
 \phi^{ij}(v- \td v) F_{\rho, i}(v) ( \pa_j f +  \lam_j f)(\td v) d \td v d v 
 - \int_{\R^3} \mu(v) 
 \phi^{ij}(v- \td v)  F_{\rho, i}(v) ( 2 \td v_j -  \lam_j ) f(\td v) d \td v d v \\
 & \teq I_{nloc, 1} + I_{nloc, 2} .
 \eal
\]


Our goal is to compare $I_{nloc}$ with \eqref{eq:coer_ansatz3} and $M_{\rho} = M_{\e}$ in \eqref{eq:Jrho_comp0},\eqref{eq:E1_loss_loc2}
\beq\label{eq:nloc_form2}
 I_{nloc, \e} \teq  
- \f{1}{2}\int \int M_{\e}(v, \td v)( F_{\rho, i}(v) F_{\rho,j}(\td v) +F_{\rho, i}( \td v) F_{\rho,j}( v)  ) \phi^{ij}(v - \td v) ,
\eeq
and show that the difference $I_{nloc} - I_{nloc, \e}$ is very small.


For $I_{nloc, 1}$, symmetrizing the integrand, using $ \pa_j f +  \lam_j f 
= \rho^{-1}  F_{\rho, j} $ \eqref{eq:Jrho_comp0}, $\phi^{ij} = \phi^{ji}$ and
\[
\phi^{ij}(v - \td v) F_{\rho,i}(v) F_{\rho, j}(\td v)
 = \phi^{ij}(v - \td v) F_{\rho,j}(v) F_{\rho, i}(\td v) ,
\]
we yield 
\[
\bal
I_{nloc, 1}  & =  - \int  \mu(v) \rho^{-1}(\td v) \phi^{ij}(v - \td v)
F_{\rho, i}(v) F_{\rho, j}(\td v) d v d \td v \\
& = - \f{1}{4} \int_{\R^3} ( \mu(v) \rho^{-1}(\td v) + \mu(\td v) \rho^{-1}(v) ) 
\phi^{ij}(v - \td v)
( F_{\rho, i}(v) F_{\rho, j}(\td v) + F_{\rho, i}(\td v) F_{\rho, j}( v)  )  d v d \td v, 
\eal
\]
which has the same form as \eqref{eq:coer_ansatz3} with a kernel $ M(v, \td v )$ \eqref{eq:Jrho_comp0} slightly different from \eqref{eq:Jrho_comp3}.

Therefore, the difference $I_{nloc} - I_{nloc, \e}$ \eqref{eq:nloc_form1},\eqref{eq:nloc_form2} consists of two error terms 
\beq\label{eq:nloc_T1}
\bal
 I_1  & = \int_{\R^3}  \B(  \phi^{ij} \ast [ (2 v_j - \lam_j ) f] \B)  \cdot  F_{ i}(v)
 \cdot (\rho \mu)(v) d v  , \\
I_2 & = - \f{1}{4} \int  ( \mu(v) \rho^{-1}(\td v) + \mu(\td v) \rho^{-1}(v) ) 
(\d_{\e}(|v|) + \d_{\e}( |\td v| ) \phi^{ij}(v - \td v)
( F_{\rho, i}(v) F_{\rho, j}(\td v) + F_{\rho, i}(\td v) F_{\rho, j}( v)  )  d v d \td v \\
& = - \int (\d_{\e}(|v|) + \d_{\e}( |\td v| ) ) \mu(v) \rho^{-1}(\td v) 
\phi^{ij}(v - \td v)  F_{\rho, i}(v) F_{\rho, j}(\td v ) d v d \td v \\
& = - \int  (\d_{\e}(|v|) + \d_{\e}( |\td v|))
( \mu \rho)(v) \phi^{ij}(v - \td v) F_i(v) F_j(\td v) d v d \td v ,
 \eal
\eeq
where we have used $F_{\rho,i} = \rho F_i$ in the last identity.

Below, we estimate these error terms. We need the following decay estimates from the kernel.
\begin{lem}\label{lem:conv1}
Denote $ \la x \ra = ( 1 + |x|^2)^{1/2}$. Suppose that $f : \R^3 \to \R, F : \R \to \R_+$ , $|f(v)| \leq F(|v|)$ for any $v$, 
and $\int_{|v| = r} f(v) \phi^{ij}(v ) d \s = 0 $ for any $r >0$.

(a) If in addition, $f(v) = 0, |v| \leq R$ for some $R \geq 1$, 
then for $ l > 1$, we get
\[
| \phi^{ij} \ast f| 
\les ( |l-1|^{-1/2} +  | l-5 |^{-1/2} )  ||  \la r \ra^{ l /2} F ||_2  \max( |v|, R)^{ - m_l } |v|,
\quad m_l = \min(2, \f{l -1}{2} ).
\]

(b) If $f$ has full support, i.e. $R = 0$, 
for $l > 3$, we get 
\[
| \phi^{ij} \ast f| \les 
 |l-3|^{-1/2 }  || F r  \la r \ra^{l /2} ||_2  |v|^{1/2} \la v \ra^{ -3/2}.
\]

\end{lem}

\begin{proof}

Denote $v_R \teq \max(2 |v|, R)$. Using $\int_{|v| = r} f(v) \phi^{ij}(v ) d \s = 0$, we get 
\[
 \phi^{ij} \ast f(v) = \int_{\R^3} \phi^{ij}(v - \td v) f( \td v ) 
  = \int_{|\td v| \geq v_R} ( \phi^{ij}(v - \td v)  - \phi^{ij}(\td v) ) f(\td v)
  + \int_{ R \leq |\td v| \leq  2|v| } \phi^{ij}(v - \td v)  f(\td v) \teq T_1 + T_2.
  \]
  For $T_1$, we yield 
  \[
  \bal
|T_1| & \les \int_{|\td v| \geq v_R} \f{|v|}{|\td v|^2}  |f(\td v)| d \td v
\leq |v| \int_{r \geq v_R} | F(r) | dr 
\les |v| \cdot || F \la r \ra^{l /2}||_2  ( \int_{r \geq v_R} r^{- l }  )^{1/2} \\
& \les (l -1)^{-1/2} |v|  v_R^{ (1-l )/2} || F \la r \ra^{l /2}||_2   .
\eal
  \]
  Since $R \geq 1$, we have $v_R^{ (1-l )/2} \les v_R^{-m_l}$.

For $T_2$, we further decompose the domain of the integral into $Q_1 = \{ |\td v| \leq |v|/2 \}, 
Q_2 = \{ |v|/2 \leq |\td v|\leq 2 |v| \} $. If $|v| \leq R/2$, we get $I_2 = 0$.

Next, we consider $|v| \geq R/2$. We have $|v|\asymp v_R \geq 1$.  For $\td v \in Q_1$, we get $|\phi^{ij}(v-\td v)| \les |v|^{-1}$and 
\[
\bal
T_{21} \teq & \int_{ Q_1} |f(\td v)\phi^{ij}(v- \td v) | d \td v 
\les |v|^{-1} \int_{Q_1} |f(\td v)| d \td v
\les  |v|^{-1} \int_0^{|v|/2} F(r) r^2 d r  \\
 \les &
|v|^{-1} || F \la r\ra^{l /2} ||_2 (\int_0^{|v|/2} r^4 \la r \ra^{-l})^{1/2} 
\les |l-5|^{-1/2} |v|^{-1} \max( |v|^{(5-l)/2}, 1) || F \la r\ra^{l/2} ||_2 \\
\les &  |l-5|^{-1/2} |v| \max( |v|^{(1-l)/2}, |v|^{-2}) || F \la r\ra^{l/2} ||_2 
\les |l-5|^{-1/2} |v|  v_R^{-m_l} || F \la r\ra^{l/2} ||_2 
\eal
 \]
 Note that the upper bound is trivial for $l =5$ and can grow in $|v|$ for $l < 3$. 



For $\td v \in Q_2$, we yield $|v-\td v| \leq 3 |v|$
\[
\bal
\int_{ Q_1} |f(\td v)\phi^{ij}(v- \td v) | d \td v 
& \les (\int_{|v-\td v| \leq 3 | v | } |v-\td v|^{-2} d \td v  \cdot 
\int_{ |v|/2 \leq |\td v| \leq 2 |v|} f^2 d \td v )^{1/2}
\les |v|^{1/2} ( \int_{ |v|/2}^{ 2 |v|} F^2 r^2 d r)^{1/2} \\
& \les |v|^{1/2} || F \la r \ra^{l /2}||_2 |v|^{ (2- l )/2} 
\les |v| \max(|v|, R)^{-( l -1)/2} || F \la r \ra^{l /2}||_2 .
\eal
\]
Summarizing the above estimate, we prove the first estimate.

If $R = 0$, for fixed $v$, we decompose the integral into 
\[
Q_1 = \{ |\td v| \leq |v|/2\}
, \quad Q_2 =  \{ |v|/2 \leq |\td v| \leq 2 |v| \} , 
\quad  Q_3 = \{  |\td v| \geq 2 |v|   \}. 
\]

Firstly, for any $ r > 0$ and $l > -1$, we note that 
\[
\int_r^{\infty} |F(s)| ds 
\leq || F r  \la r \ra^{l/2} ||_2 
( \int_r^{\infty} s^{-2} \la s \ra^{-l} ds)^{1/2}
\les (l +1)^{-1/2}  || F r \la r \ra^{l /2} ||_2 r^{-1/2} \la r \ra^{- l /2} .
\]
If $r < 1$, the integral is bounded by $C r^{-1}$; if $r > 1$, the integral is bounded by $
( l +1)^{-1} r^{-l -1}$, which implies the above estimate.


The integral in $Q_3$ can be bounded similarly 
\[
\bal
| \int_{Q_3} (\phi^{ij}(v - \td v) - \phi^{ij}(\td v) ) f(\td v ) d \td v|
& \les |v| \int_{ |\td v| \geq 2 |v|} |\td v|^{-2} | f(\td v) | d \td v 
\les |v| \int_{ 2 |v|}^{\infty} F(r) dr  \\
& \les (l +1)^{-1/2} |v|^{1/2} \la |v| \ra^{-l /2}  || F r \la r \ra^{l /2} ||_2 .
\eal
\]

In $Q_1$, we get $|\phi^{ij}(v- \td v)| \les |v|^{-1}$. Since $l > 3$, it follows 
\[
\bal
& | \int_{Q_1} \phi^{ij}(v - \td v)   f(\td v ) d \td v|
 \les |v|^{-1} \int_{|\td v| \leq |v| / 2} |f(\td v)| d \td v
 = |v|^{-1} \int_{0}^{|v|/2}  |F(r)| r^2 dr \\
  \les & | v|^{-1} || r  \la r \ra^{l /2} F ||_2  ( \int_0^{|v|} r^2 \la r \ra^{-l } d r )^{1/2}  \les | v|^{-1}  || r  \la r \ra^{l /2} F ||_2 \B( \one_{|v|\leq 1} |v|^{3/2}
 + \one_{|v| > 1}  |l-3|^{-1/2}  \B)   \\
  \les & (l-3)^{-1/2} |v|^{1/2} \la |v|\ra^{-3/2}  || r  \la r \ra^{l/2} F ||_2 .
\eal
\]

In $Q_2$, we get $ | \phi^{ij}(v - \td v ) |  \les |v- \td v|^{-1}, |v| \asymp |\td v |$. We get 
\[
\bal
| \int_{Q_2} \phi^{ij}(v - \td v)   f(\td v ) d \td v|
& \les  \B( \int_{Q_2} |v - \td v|^{-2} d \td v \cdot
\int_{Q_2} |f(\td v)|^2 d \td v 
\B)^{1/2}
\les  |v|^{1/2} ( \int_{|v|/2}^{2|v|} F(r)^2 r^2 d r)^{1/2} \\
& \les |v|^{1/2} ||  r \la r \ra^{l/2} F ||_2   \la |v| \ra^{-l/2}
\les |v|^{1/2} \la |v| \ra^{-l/2}  || F r \la r \ra^{l/2} ||_2 .
\eal 
\]

Combining the above estimates, we prove the second inequality.
\end{proof}

\subsubsection{Estimate of $I_1$}

Recall $I_1$ \eqref{eq:nloc_T1}. Denote $q_j(v) =(2 v_j - \lam_j ) f(|v|)  $. 
From \eqref{eq:nota2}, \eqref{eq:lin_op_lam},
\beq\label{eq:E1_loss_nloc0}
\rho = e^{r^2}, \quad \lam_j = 2 v_j \rho, \  r \leq R_1, \quad \lam_j =\f{v_j}{|v|} \lam ,
\quad  0\leq \lam(r) ,|\lam_j| \leq 2 r,
 \quad \rho \mu \leq 1, \ r \geq 0,
\eeq
and \eqref{eq:Jrho_comp0}, \eqref{eq:Jrho_comp02} for $F_i, F_0$, we get 
\beq\label{eq:E1_loss_nloc01}
q_j(v) = 0, \quad |v| \leq R_1, \quad |q_j(v)| \leq 2  |v| \cdot  |f|,
\quad |F_j| \leq |F_0| \les |\pa_r f | + |r f| , \ \forall  \ v.
\eeq

Next, we control $I_1$ by $ || r ( |f| + |\pa_r f | )\la r \ra^{l/2}||_2 
 , l \in (7, 20)$.  Due to the odd symmetry, we get
\[
\int_{|v| = r } \phi^{ij}(v) \f{v_j}{|v|} d \s(v) = 0. 
\]

Applying Lemma \ref{lem:conv1} with $ l \in (7, 20)$, we obtain $ m_l = \min(2, \f{l-1}{2} ) = 2$,
\[
\bal
 |\phi^{ij}(v) \ast q_j |
 \les    || q_j \la r \ra^{l/2} ||_2  \max( |v|, R_1)^{ -m_l} |v|  \les     ||  r f \la r \ra^{l/2 } ||_2  \max( |v|, R_1)^{ - 2 } |v| ,
\eal
\]

It follows 
\[
\bal
 & | I_1  |  \les  || r f \la r \ra^{l/2 } ||_2  \int |F_0| |v| \max( |v|, R_1)^{ - 2} \rho \mu(v) d v 
\les || r f \la r \ra^{l/2} ||_2  \int ( |\pa_r f | + | r f  | )  r^3 \max(r, R_1)^{-2} d r ,
\eal
\]
where we have used $\rho \mu \leq 1$, $|F_0| \les |\pa_r f| + |r f|$ \eqref{eq:E1_loss_nloc01}, and rewritten the integral in $1D$ in the second inequality. 
Applying Cauchy-Schwarz inequality, we obtain
\beq\label{eq:E1_loss_nloc1}
\bal 
 I_1 & \les  || r (|f| + |\pa_r f|) \la r \ra^{\f{l}{2}} ||_2^2 
 \cdot || (r^2 + r^3) \la r \ra^{ -{\f{l}{2}}  } \max(r, R_1)^{-2} ||_2 \\
 &   \les  |l-7|^{-{\f{l}{2}} } R_1^{-2} || r (|f| + |\pa_r f|) \la r \ra^{\f{l}{2}} ||_2^2 , \quad l \in (7, 20),
 \eal
\eeq
where we have used the following estimate for the integral with $l > 7 $, 
\[
\int_0^{R_1} \la r \ra^{6 - l} R_1^{-4}  d r 
+ \int_{R_1}^{\infty} \la r \ra^{6 - l -4} d r
\les R_1^{ 3-l} + |l-7|^{-1} \max( R_1^{7-l}, 1) R_1^{-4} 
\les |l-7|^{-1} R_1^{-4}.
\]



\subsubsection{Estimate of $I_2$}

Firstly, we rewrite $I_2$ \eqref{eq:nloc_T1} as follows 
\[
I_2 = - \int \mu \rho \d_{\e}  F_i (\phi^{ij} \ast F_j)  d v
- \int \mu \rho  F_i   ( \phi^{ij} \ast ( F_j  \d_{\e} ) ) d v  
= - \int (\phi^{ij} \ast  ( \mu \rho \d_{\e}  F_i ) )   F_j   d v
- \int \mu \rho  F_i   ( \phi^{ij} \ast ( F_j  \d_{\e} ) ) d v  
\]

Recall from \eqref{eq:Jrho_comp02}  that
\[
F_i = \f{v_i}{|v|} (\pa_r f + \lam f) =\f{v_i}{|v|} (\pa_r f + \f{\pa_r \rho}{\rho} f ) =
\f{v_i}{|v|} F_0 .
\]

Using the second estimate in Lemma \ref{lem:conv1} with $ l = 5$, and $\mu \rho \leq 1$, 
$|F_i| \leq F_0$ \eqref{eq:Jrho_comp0},\eqref{eq:Jrho_comp02}, we obtain 
\beq\label{eq:E1_loss_nloc2}
\bal
|I_2|
& \les  || F_0 r \d_{\e} \la r \ra^{ l /2} ||_2
 \int_{\R^3}  | F_0|  |v|^{1/2} \la |v| \ra^{-3/2} d v  \\
 \eal
\eeq

Using the relation between radial integral and integrals in $\R^3$, we get
\[
 \int_{\R^3}  | F_0|  |v|^{1/2} \la |v| \ra^{-3/2} d v 
 \les \int |F_0(r)| \la r \ra^{-3/2} r^{5/2} d r
 \les  || F_0 r \la r \ra^{2/3}||_2 || \la r \ra^{-2/3}||_2 
 \les || F_0 r \la r \ra^{2/3}||_2.
\]

Using the smallness of $\d_{\e}$, the rapid decays of $\d_{\e}$ \eqref{eq:gam_ineq1}, and \eqref{eq:E1_loss_nloc01}
\[
\d_{\e}  \les e^{-R_1^2/2} \la r \ra^{-100}, \quad |F_0| \les |\pa_r f | + 2 r |f|,
\]
and estimates similar to \eqref{eq:E1_loss_nloc1}, we establish
\beq\label{eq:E1_loss_nloc3}
|I_2| \les e^{-  \f{ R_1^2 }{2} } || F_0 r \la r \ra^{2/3} ||_2 || F_0 r \la r \ra^{-10} ||_2 
\les e^{- \f{ R_1^2 }{2} } || r ( |\pa_r f | + | f|) \la r \ra^{l/2} ||_2,  \  l = 5.
\eeq

Combining  \eqref{eq:nloc_T1},\eqref{eq:E1_loss_nloc1},\eqref{eq:E1_loss_nloc3}, we obtain the estimate of $I_{nloc} - I_{nloc, \e}$ \eqref{eq:nloc_form1}, \eqref{eq:nloc_form2}
\[
  | I_{nloc} - I_{nloc, \e}| \leq |I_1| + |I_2| \les  C R_1^{-2} |l-7|^{-\f{1}{2}} \cdot || r ( |\pa_r f | + | f|) \la r \ra^{l/2} ||_2 ,
\]
which along with \eqref{eq:E1_loss_loc2}, and $\int_{ R_1}^{\infty} f^2 r^2 d r \les || r ( |\pa_r f | + | f|) \la r \ra^{l/2} ||_2^2$ implies
\beq\label{eq:lin_coer}
\bal
 & - \la (\cL_{loc} + \cL_{nloc} ) f,  f \rho \ra
  \geq Q_{\rho} 
 -  C R_1^{-2} |l-7|^{-\f{1}{2}} \cdot || r ( |\pa_r f | + | f|) \la r \ra^{l/2} ||_2, \   l \in (7, 20), \\
 & Q_{\rho} \teq \f{1}{2}\int  \phi^{ij}(v - \td v)  M_{\e}(v, \td v)  ( F_{\rho, i}(v) - F_{\rho, i}( \td v) )
( F_{\rho, j}(v) - F_{\rho, j}(\td v)) d \td v d v  \\ 
\eal
\eeq
and justify  \eqref{eq:coer_ansatz}. We will further bound these error terms using the energy in Section \ref{sec:E1_lower}.

\subsection{Localized coercive estimates}

We want to obtain a coercive estimate of $ - \la \cL_1 f , f \rho \ra$ using the relation \eqref{eq:coer_ansatz} and following \cite{guo2002landau}. Since the proof in \cite{guo2002landau} is based on a compactness argument, if we apply it directly to \eqref{eq:coer_ansatz}, the constant in the coercive bound depends on $R_1$ via the weight $\rho$. 
Below, we develop localized coercive estimates with a spectral gap independent of $R_1$.
We note that the explicit constant in the $L^2(\mu^{-1})$ coercive estimates of $\cL_1$ in the whole space has been obtained in \cite{mouhot2006explicit}. It is not clearly if such an argument can be localized with a constant independent of $R_1$.

From the definition of \eqref{eq:gam_ineq1}, by choosing $\e$ first and then $R_1$ large enough, we can obtain 
\[
|| \d_{\e}||_{\infty} < \f{1}{4}, \quad  1 - \d_{\e}(v) - \d_{\e}(\td v) \geq \f{1}{2}.
\]

Recall the definitions of $\lam, \lam_i$ \eqref{eq:lin_op_lam}, $M(v, \td v), M_{\e}(v, \td v)$ from \eqref{eq:Jrho_comp0}, \eqref{eq:Jrho_comp3}, \eqref{eq:E1_loss_loc2}.
For $|v|, |\td v| \leq R_1$, we get
\[
\bal
 & \rho(v) = \mu^{-1}(v) = e^{|v|^2}, \quad |v| \leq R_1,  
 \quad \lam = \f{\pa_r \rho}{\rho} = 2 |v|,  \quad \lam_i = 2 v_i,  \\
 & M_{\e}(v, \td v) \geq \f{1}{2} M(v, \td v) =\f{1}{2} (\rho(v)\rho(\td v))^{-1}
   = \f{1}{2} \mu(v) \mu(\td v).
   \eal
\]

Using the above estimates and localizing the integral in \eqref{eq:coer_ansatz} to $|v|, |\td v| \leq R_1$, we yield 
\[
\bal
  Q_{\rho} & \geq \f{1}{2} Q(f, R_1), \quad F_{\mu^{-1}, i} = \pa_i ( f \mu^{-1} ) ,\\
  Q(f, r) & \teq \f{1}{2} \int_{|v|, |\td v| \leq R_1} \mu(v) \mu(\td v) 
\phi^{ij}(v - \td v)  ( F_{\mu^{-1}, i}(v) - F_{\mu^{-1}, i}(\td v)  )
( F_{\mu^{-1}, j}(v) - F_{\mu^{-1}, j}(\td v)  ) d v d \td v,
\eal
\]
where we have used the notation \eqref{eq:Jrho_comp0} $F_{\rho, i}$  with $\rho(v) = \mu^{-1}(v), |v| \leq R_1$.

The integrand in $Q(f, r)$ is the same as that in Lemma 4 in \cite{guo2002landau}. It is clear that $Q(f, r) = 0$ if and only if 
$(f \mu^{-1})(v) =  a + b \cdot v + c|v|^2, |v| \leq r$ for some constants $a,b,c$.

We generalize the coercive estimate of $\cL_1$ in Lemma 5 in \cite{guo2002landau} to the following uniformly localized estimate. Lemma 5 in \cite{guo2002landau} implies the case of $n = \infty$ below.

\begin{lem}\label{lem:coer_unif}

For any $n  \in Z_+, n \geq 1$,  
there exists $\d_n > 0$ such that for $f$ radially symmetric, we have
\beq\label{eq:coer_unif}
Q(f, n) \geq \d_n | \{ I - \Pi_0(n) \} ( f \mu^{-1/2}) |_{\s, n}^2, 
\eeq
where $\Pi_0(r)$ is the $L^2$ projection 
\[
\Pi_0(r) p \teq  \sum_{1\leq j\leq 5} e_j \cdot  \int_{|v| \leq r} p e_j d v ,
\]
$\{ e_j \}_{j \leq 5}$ is the orthonormal basis for $\{ 1, v, |v|^2 \} \mu^{1/2} \one_{|v|\leq r}$, and $|p|_{\s, R}$ is the localized norm 
\[
|p|_{\s, R} \teq \int_{|v| \leq R} \s_R^{ij} ( \pa_i p \pa_j p  + v_i v_j p^2 ) d v,
\quad \s_R^{ij}(v) \teq  \int_{ | \td v| \leq R } \phi^{ij}(v - \td v) \mu(\td v) d \td v 
\]
Moreover, $\d_n$ is uniformly bounded from below. For some $\d_*>0$, we have 
\beq\label{eq:unif_gap}
 \liminf_{n \to \infty} \d_n \geq \min(1, \d_{\infty}), \quad  \d_n \geq \d_* > 0.
\eeq

\end{lem}

Below, we restrict the discussion to $f$ being a radially symmetric function.

\begin{proof}

Denote $ p = f \mu^{-1/2}$. Firstly, using $F_{\mu^{-1}, i} = \pa_i( f \mu^{-1}) 
= \pa_i( p \mu^{-1/2}) = \mu^{-1/2}( \pa_i p + v_i p )$, we have the identity 
\beq\label{eq:coer_comp1}
\bal
Q(f, n) & = |p|_{\s, n}^2 + \int_{|v|\leq R} 2 \s^{ij}_R \pa_i p  \cdot  v_j p d v - \G(p, n), \\
\G(p, n) &=  \f{1}{2} \int_{|v| , |\td v| \leq n} 
\phi^{ij}(v - \td v) \mu(v) \mu(\td v) ( F_{\mu^{-1}, i}( v) F_{\mu^{-1}, j}(\td v)
+ F_{\mu^{-1}, i}( \td v) F_{\mu^{-1}, j}(  v)) d v  d \td v  \\
& = \int_{|v| , |\td v| \leq n}  
\phi^{ij}(v - \td v) \mu(v) \mu(\td v) \pa_i (p \mu^{-1/2} )(v) 
\cdot 
\pa_j  (p \mu^{-1/2} )(\td v)  d v d \td v,
\eal
\eeq
where the first two terms in $Q(f, n)$ comes from the diagonal part, and $\G$ comes from the cross terms in the double integral in $Q(f, n)$.

For each $R = n \in Z_+$, following the compactness argument in \cite{guo2002landau}, we can establish \eqref{eq:coer_unif} with constant $\d_n > 0$ depending on $n$. Note that 
for each $n$, the domain of the integral is bounded. The proof here is even easier and we omit the details. If $p_m$ converges to $p_{\infty}$ weakly in the inner product associated with  $| \cdot|_{\s, n}$, $\na p_m$ converges to $\na p_{\infty}$ weakly in $L^2$ and $v p_m$ strongly to $v p_{\infty}$ in $L^2$. As a result, the second and the third term in $Q(p_m, n)$ converge. See the estimates below related to these convergences. In particular, the estimate \eqref{eq:coer_unif} holds for radially symmetric functions.

Next, we prove the limit in \eqref{eq:unif_gap}, which implies that $\d_n$ is uniformly bounded from below. We follow the compactness argument in \cite{guo2002landau}. 
Assuming the contrary, we have a sequence of radially symmetric functions  $p_{n_k}$ with $n_k \to \infty$, $f_{n_k} = p_{n_k}\mu^{1/2}$ and some $ \d <\min(1, \d_{\infty})$ satisfying
\beq\label{eq:coer_comp12}
|p_{n_k}|_{n_k, \s} = 1, \quad 0\leq Q(f_{n_k}, n_k ) \leq \d < \min(1, \d_{\infty}), 
\quad \int_{|v|\leq n_k} f_{n_k} h d v =0, \ h = 1, v_i, |v|^2. 
\eeq

\vspace{0.1in}
\paragraph{\bf{Equivalence of norms}}
Firstly, it is easy to see that 
\beq\label{eq:coer_comp20}
\s_{n}^{ij}(v)  = \int_{|v|\leq n} \phi^{ij}(v- \td v) \mu(\td v) d \td v \to \s^{ij}(v),
\eeq
and the convergence is uniform in any compact set of $v$. For a radially symmetric function $p$, 
using \eqref{eq:int_rad} and \eqref{eq:Gauss_id2}, we get that the norm $|p|_{\s, R}$ is equivalent to 
\beq\label{eq:coer_equiv}
  \int_0^R  - \bar g_{R, rr} r^2 ( (\pa_r p )^2 + r^2 p^2 ) dr ,
  \quad 
  \bar g_{R, rr} \teq \pa_{rr}(-\D)^{-2}(\mu \one_{|v| \leq R} )
   = - \s_R^{ij} \f{v_i v_j}{|v|^2}, \quad |v| = r.
\eeq

For $ R \geq 1$, from \eqref{eq:BSlaw3}, we get that $- \bar g_{R, rr}$ is increasing in $R$,
\beq\label{eq:coer_comp2}
 (1 + r)^{-3} \asymp  - \bar g_{R, rr} , \quad \lim_{R \to \infty} \bar g_{R, rr} = \bar g_{rr}.
\eeq

Let $\chi$ be a smooth cutoff function with $\chi = 1$ on $r \leq 1$ and $\chi = 0$ for $r \geq 2$. 
Using the Poincare inequality, for $R \geq 2$, we yield 
\[
||  (1 + r^2)^{-1/4} r p  \one_{r \leq R} ||_2
\les || \chi r p   ||_2 + ||  (1-\chi) r^{1/2} p  \one_{r \leq R}  ||_2 
\les \int_0^R (- \bar g_{R, rr} r^2) ( (\pa_r p )^2 + r^2 p^2) dr.
\]

Using integration by parts, we get 
\[
\bal 
T & = \int p^2 \chi^2 d r 
=  - \int (p^2 \chi^2)_r  r  dr 
= - 2 \int (r p_r) p \chi^2  d r 
 - 2\int p^2 \chi_r \chi r  \\
& \les T^{1/2} ( || r p_r \chi ||_2 
+ || p \chi_r r ||_2 )
\les T^{1/2}  \B( \int_0^2 r^2 ( (\pa_r p )^2 + r^2 p^2) d r \B)^{1/2} ,
\eal
\]
where we have used $|\chi_r| \les \one_{ r \in [1, 2]}\les r$. It follows 
\beq\label{eq:coer_norm}
 \int_0^1 p^2 d r
 \leq T \les 
 \int_0^2 r^2 ( (\pa_r p )^2 + r^2 p^2) d r 
\les  \int_0^2 (-\bar g_{R, rr}r^2) ( (\pa_r p )^2 + r^2 p^2) d r .
\eeq

\vspace{0.1in}
\paragraph{\bf{Convergence}}

For a fixed $m \in \Z_+$ and $n\geq m$, since $- \bar g_{R, rr}$ is increasing in $R$ 
\eqref{eq:coer_equiv}, \eqref{eq:BSlaw3}, we get 
\[ 
|p_n|_{\s, m} \leq |p_n|_{\s, n} \leq 1.
\]
Due to the convergence of $\s_n^{ij}$ \eqref{eq:coer_comp20}, \eqref{eq:coer_comp2}, \eqref{eq:coer_norm}, the $L^2$ boundedness of $ r \pa_r p_{n_k}, (1 + r) p_{n_k}$,  and the bound \eqref{eq:coer_comp2}, we get a subsequence such that $ r \pa_r p_{n_k}$ converges weakly to $r \pa_r p_0$ in $L^2( [0, m] )$ and $ (1+r) p_{n_k}$ converges weakly to $(1+r) p_0$. Since $\pa_r (r p_{n_k}), r p_{n_k}$ converges weakly to $\pa_r (r p_0), r p_0$ in $L^2( [0, m] )$ respectively, using Ascoli-Arzela theorem, we get that $r p_{n_k}$ is continuous and converges to a continuous function $r p_0$ uniformly in $[0, m]$ and in $L^{\infty}([0, m])$. Using a diagonal argument, we can extract another subsequence $\{ p_{n_k} \}_{k\geq 1}$ (still denotes as $p_{n_k}$ for simplicity) of $ \{ p_{n_i} \}_{i \geq 1}$ such that 
\beq\label{eq:coer_comp31}
| p_{n_k}|_{\s,  n_k } = 1, \quad  \pa_r ( r p_{n_k} ) \rightharpoonup  \pa_r(  r p_0  )
\eeq
weakly in $L^2( [0, m])$, and 
\beq\label{eq:coer_comp32}
r p_{n_k} \to r p_0
\eeq
uniformly in $L^{\infty}([0, m])$ for any $ m \in \Z_+$. It is not difficult to see that the limiting functions agree for different $m$, and then $p_0$ is globally defined.
Moreover, using Fatou's theorem, we yield $| p_0|_{\s, m} \leq 1$. Taking $m \to \infty$, we get 
\[
 | p_0|_{\s, \infty} \leq 1.
\]

Next, we apply the identity \eqref{eq:coer_comp1} to $p_{n_k}, R = n_k, f_{n_k} = \mu^{1/2} p_{n_k}$ and use \eqref{eq:int_rad} to get
\beq\label{eq:coer_comp4}
Q( f_{n_k}, n_k) = 1 +  4 \pi \cdot 2  \int_0^{n_k} ( - \bar g_{n_k, rr}) r^2 \cdot \pa_r p_{n_k} \cdot r p_{n_k} d r
- \G( p_{n_k}, n_k) ,
\eeq
and want to pass the limit to $p_0$. For the second term, using integration by parts, we yield 
\beq\label{eq:coer_comp42}
I_2(n_k) \teq  - 4 \pi \cdot  \bar g_{n_k, rr} r^3 p^2_{n_k} \B|_0^{n_k} + 4\pi \cdot \int_0^{n_k} ( \bar g_{n_k, rr} r^3 )_r p_{n_k}^2  \teq I_{21}(n_k) + I_{22}(n_k).
\eeq

At $r=0$, the boundary term vanishes since $r p_{n_k}$ is bounded. We get the crucial non-negativity 
\beq\label{eq:coer_comp43}
I_{21}(n_k) = - 4 \pi \cdot \bar g_{n_k, rr} r^3 p_{n_k}^2 |_{r = n_k} \geq 0.
\eeq
From \eqref{eq:BSlaw3} and \eqref{eq:coer_comp2}, it is clear that 
\[
  |( \bar g_{n_k, r^2} r^3)_r| 
\les | r^2 B_1( \mu \one_{r \leq R}) |
  \les r^{-1} \les r^{-1}  |\bar g_{n_k, r^2} r^4 |, \quad r \geq 1.
\]
 In fact, $ (r^3 \bar g_{R, rr})_r$ decays exponentially fast in $r$ uniformly in $R$ but we do not need this faster decay rate. 
Since for any fixed $m$, $r p_{n_k} \to r p_0$, $\bar g_{n_k, rr} \to \bar g_{ rr}$ uniformly in $L^{\infty}[0, m]$, and 
\[
 \int_m^{n_k}   |( \bar g_{n_k, r^2} r^3)_r p_{n_k}^2| 
 \les m^{-1} | p_{n_k}|_{\s, n_k} \les m^{-1},
\]
we yield 
\beq\label{eq:coer_comp44}
\lim_{k\to \infty} I_{22}(n_k)= 4\pi \cdot \int_0^{\infty} (\bar g_{rr} r^3)_r p_0^2 .
\eeq

Next, we study the convergence of $\G( p_{n_k}, n_k)$. We can reformulate $\G$ as follows 
\[
\G(p, n) = \int  ( \phi^{ij} \ast ( \mu^{1/2} P_{n,i} ) ) \mu^{1/2} P_{n,j}(v) d v  , \quad 
P_{n,i} \teq  (\pa_i p + v_i p ) \one_{ |v| \leq n}
= \f{v_i}{|v|} ( \pa_r p + r p) \one_{ |v| \leq n}, 
\]
where $r = |v|$.  Let $P_{n_k, i}$ be the function associated to $p_{n_k}$. Using the equivalence \eqref{eq:coer_equiv} and the boundedness \eqref{eq:coer_comp31}, we get $ (1 + |v|)^{-3/2} P_{n_k, i} \in L^2(\R^3)$. Using the convergence \eqref{eq:coer_comp31}-\eqref{eq:coer_comp32}, 
for any fixed $R > 0$, we get 
\[
 P_{n_k, i} \one_{|v| \leq R} \rightharpoonup (\pa_i p_0( v ) + v_i p_0( v ) ) \one_{|v| \leq R}
\]
in $L^2(\R^3)$. Note that passing from the integral in $L^2(\R^3)$ to $L^2(\R_+)$, we get $r^2$ from the volume measure $r^2 d r$ \eqref{eq:int_rad}. Moreover, for any $R > 0$, using H\"older's inequality, we get
\[
\bal
   |\phi^{ij} \ast (\mu^{1/2} P_{n_k, i}| \one_{|v| > R} )|
  & \les \B( \int_{\R^3} |v-\td v|^{-2} \mu^{1/8}(\td v) d \td v \B)^{1/2}
  || \mu^{1/4} \one_{|v| > R}||_{\infty}
 || \mu^{1/8} P_{n_k, i} \one_{|v| > R}||_2 \\
 & \les e^{-R^2/ 4} (1 + |v|)^{-1},
 \eal
\]
where the decay of $v$ comes from the first term on the right hand side. The same estimate applies to $ \pa_i p_0 + v_i p_0$. For each $v$, since $\phi^{ij}(v - \td v) \mu(\td v)^{1/2} \in L^2(\R^3)$, using the above weak convergence of $ P_{n_k, i} \one_{|v| \leq R}$ and the smallness of $ \mu^{1/2} P_{n_k, i} \one_{|v| > R}$, we obtain the pointwise convergence 
\[
\phi^{ij} \ast (\mu^{1/2} P_{n_k, i}) \to \phi^{ij} \ast ( \mu^{1/2} (\pa_i p_0 + v_i p_0)  ).
\]

Taking $ R = 0$ in the above estimate, we obtain
\[
\mu^{1/4}    |\phi^{ij} \ast (\mu^{1/2} P_{n_k, i} )|  
\les  \exp^{-|v|^2 / 4 } (1 + |v|)^{-1}
\]
which is in $L^1(\R^3) \cap L^{\infty}(\R^3)$. Using dominated convergence theorem, we yield strong convergence
\[
 \mu^{1/4} ( \phi^{ij} \ast (\mu^{1/2} P_{n_k, i}) ) \to 
  \mu^{1/4} ( \phi^{ij} \ast ( \mu^{1/2} (\pa_i p_0 + v_i p_0)  ) )
\]
in $L^2(\R^3)$. Since for any fixed $R \geq 1$, $\mu^{1/4} P_{n_k, j} $ converges weakly to 
$\mu^{1/4} (\pa_j p_0 + v_j p_0)$ in $L^2( |v| \leq R)$, and the tail is small
\[
|| \mu^{1/4} P_{n_k, j} \one_{|v| > R} ||_2 
\leq \mu^{1/8}(R) || \mu^{1/8} P_{n_k, j} \one_{|v| > R} ||_2 
\les e^{-R^2/8}, 
\]
we prove 
\[
\bal
 \lim_{k\to\infty }\G( p_{n_k}, n_k) &= \lim_{k\to\infty }
 \int  \mu^{1/4} ( \phi^{ij} \ast (\mu^{1/2} P_{n_k, i}) ) \cdot \mu^{1/4} P_{n_k, j} d v  \\
&= \int  \mu^{1/4} ( \phi^{ij} \ast (\mu^{1/2} (\pa_i p_0 + v_i p_0) ) ) \cdot \mu^{1/4} 
(\pa_j p_0 + v_j p_0)
 d v 
 = \G( p_0, \infty).
 \eal
\]

\vspace{0.1in}
\paragraph{\bf{The limiting case}}
Using the smallness of $Q(f_{n_k}, n_k)$ from the assumption \eqref{eq:coer_comp12}, \eqref{eq:coer_comp43}, \eqref{eq:coer_comp44}, and the above convergence, taking $\liminf$ and $k\to\infty$ in \eqref{eq:coer_comp4}, we establish 
\[
\d > \liminf_{k\to \infty} Q( f_{n_k}, n_k)
 = 1 + \liminf_{k\to\infty} I_{21}(n_k)
+4\pi \int_0^{\infty} (\bar g_{rr} r^3)_r p_0^2  - \G( p_0, \infty)
\geq 1 + 4 \pi \int_0^{\infty} (\bar g_{rr} r^3)_r p_0^2  - \G( p_0, \infty).
\]

Applying \eqref{eq:coer_comp1} to $(f, n) = ( p_0 \mu^{1/2}, \infty)$, using \eqref{eq:int_rad} and  then integration by parts, we yield 
\[
Q(f_0, \infty) = |p_0|_{\s,\infty}^2 
  + 4\pi \int_0^{\infty} (\bar g_{rr} r^3)_r p_0^2 - \G( p_0, \infty), \quad f_0 = p_0 \mu^{1/2}.
\]

The boundary term in the integration by parts vanishes since the domain of integral is $L^2(\R^3)$ and $ (1 + r)^{1/2} p_0(r) \in L^2$. Comparing the above two estimates, we yield 
\[
\d - Q(f_0, \infty) \geq (1 - |p_0|_{\s,\infty}^2), \quad 0 \geq 1 - |p_0|_{\s,\infty}^2  + Q(f_0, \infty) - \d .
\]

Due to the weak convergence, the boundedness \eqref{eq:coer_comp31}, \eqref{eq:coer_comp32},
and the fast decay of $f_0 = \mu^{1/2} p_0, f_{n_k} = \mu^{1/2} p_{n_k}$, taking $k\to\infty$ in the orthogonal conditions \eqref{eq:coer_comp12}, we obtain $\int_{\R^3} f_0 h d v  =0, h = 1, v_i, |v|^2$.
Applying \eqref{eq:coer_unif} to $(f, n) = (f_0, \infty)$, we obtain 
\[
Q(f_0,\infty) \geq \d_{\infty} | p_0|_{\s,\infty}^2.
\]
Since $0\leq \d < \min(1, \d_{\infty})$ \eqref{eq:coer_comp12}, we further obtain 
\[
0 \geq 1 - |p_0|_{\s,\infty}^2  + Q(f_0, \infty) - \d
\geq 1 - |p_0|_{\s,\infty}^2 + \d | p_0|_{\s,\infty}^2 - \d
= (1- \d) (1-| p_0|_{\s,\infty}^2)  , \quad \d < 1,
\]
Since $0\leq \d < \min(1,\d_{\infty}), |p_0|_{\s,\infty}^2 \leq 1$, the above inequalities must be equalities and
\[
|p_0|_{\s, \infty}^2 = 1, \quad \d |p_0|_{\s, \infty}^2 =  Q(f_0,\infty) \geq \d_{\infty} |p_0|_{\s, \infty}^2 .
\]
The second inequality implies $|p_0|_{\s, \infty}^2 = 0 $, which contradicts $|p_0|_{\s, \infty}^2=1$.
 We prove \eqref{eq:unif_gap}. 
\end{proof}

To apply Lemma \ref{lem:coer_unif}, we choose $R_1  \in \Z_+, R_1 \geq 10$ and yield 
\[
\bal
Q_{\rho}
& \geq \f{1}{2} Q(f , R_1)
\geq \f{ \d_*}{2} |  (I - \Pi_0(R_1)) (f \mu^{-1/2}) |_{\s, R_1} = C \d_* \int_0^{R_1} (- g_{R_1, rr} r^2) ( ( \pa_r F_{\perp} )^2 + 
  (r F_{\perp} )^2) d r  , \\
   F_{\perp} & \teq  (I - \Pi_0(R_1)) (f \mu^{-1/2}) ,
  \eal
  \]
for some absolute constant $C > 0$. 
Note that $ - g_{R_1, rr}  \gtr - \bar g_{rr} = |\bar g_{rr}|$, 
\eqref{eq:coer_comp2}. We can assume that $\d_* < 1$.
Using $\pa_r (f \mu^{-1/2}) =\mu^{-1/2} (\pa_r f + r f )$, 
\[
|\Pi_0(R_1) (f \mu^{-1/2})|_{\s, R_1} \asymp  |\int_0^{R_1} f r^2| + |\int_0^{R_1} f r^4|
\]
and the Cauchy-Schwarz inequality, \eqref{eq:coer_norm}, we get 
\beq\label{eq:coer_unif3}
Q_{\rho} \geq \d_{**} \int_0^{R_1} | \bar g_{rr} r^2| \mu^{-1} ( (\pa_r f )^2 + \la r \ra^2 f^2 ) d r
-  C \B( |\int_0^{R_1} f r^2|^2 + |\int_0^{R_1} f r^4|^2 \B)
\eeq
for another absolute constant $\d_{**} \in (0, 1/2]$ and some $C$ independent of $\d_{**}$. In fact, we can improve the constant from $C$ to $C \d_{**}$ but we do no need this extra smallness.
In Section \ref{sec:E1_lower}, we will further control the moment of $f$ so that we can control the whole norm $|f \mu^{-1/2}|_{\s, R_1}$.



\subsection{Energy estimate with a weight growing polynomially}\label{sec:E2_lin}


In this section, we perform energy estimate on $J(\rho_2)$ defined below
with some weight $\rho_2$ growing polynomially.
Our goal is to construct $\rho_2$ such that $J(\rho_2)$ is positive up to $O(1)$ loss in the near field $[0, R_0]$ for some absolute constant $R_0$. We will use the uniform coercive estimate \eqref{eq:coer_unif3} to control the $O(1)$ loss in Section \ref{sec:lin_sum}. See the ideas in Section \ref{sec:ideas}.

Recall the operator $\cL_1$ from \eqref{eq:lin_op}. For  $\rho_2$ to be determined with $\rho_2 \les r^2,r \leq 1$, we introduce
\beq\label{eq:Jrho2}
 J(\rho_2) \teq \int_0^{\infty} \cL_1 f \cdot f \rho_2 d r 
 = \int_0^{\infty} ( - f_{rr} \bar g_{rr} - 2 \f{ f_r \bar g_r }{r^2}
 + 2 f \bar f - \bar f_{rr} g_{rr} - 2 \f{\bar f_r g_r}{r^2} ) f \rho_2.
\eeq

We observe that the coefficient of the nonlocal terms, e.g. $\bar f, \bar f_r, \bar f_{rr}$ decays exponentially fast, which is much faster than that of the local terms. 
Below, we focus on the local term and will treat the nonlocal terms perturbatively  for $r \gtr 1$. 

Since $\rho_2 \les r^2$ near $r=0$, using integration by parts, we yield 
\[
J(\rho_2)
= \int \bar g_{rr} f_r^2 \rho_2 
+ \B( \pa_r ( \f{ \bar g_r \rho_2}{r^2} ) - \f{ (\bar g_{rr} \rho_2)_{rr}  }{2} + 2 \bar f \rho_2 \B) f^2
- 2 \f{ \bar f_r  }{r^2} g_r f \rho_2 - \bar f_{rr} g_{rr} f \rho_2
\]
Since $ \bar g_{rr} < 0$, the first term is a damping term. 
For the second term, from \eqref{eq:BSlaw2}, \eqref{eq:BSlaw3}, we get 
\beq\label{eq:gbar_asym}
\bal
& - \bar g_r \asymp \f{r}{1+r}, \ \bar g_r \leq 0,  
\  - \bar g_{rr} \asymp (1 + r)^{-3}, \ 
 \bar g_{rrr} \asymp \min(r, r^{-4}),
\  |\bar g_{rrrr}| \les \min(1, r^{-5}).
\eal
\eeq

Denote by $D(\rho_2)$ the coefficient of the damping term for $f^2$
\[
D(\rho_2) \teq  \pa_r ( \f{ \bar g_r \rho_2}{r^2} ) - \f{ (\bar g_{rr} \rho_2)_{rr}  }{2} - 2 \bar f \rho_2
= \bar g_r \pa_r \f{\rho_2}{r^2} + \B(  \bar g_{rr} \f{\rho_2}{r^2} - \f{ (\bar g_{rr} \rho_2)_{rr}  }{2} - 2 \bar f \rho_2  \B).
\]

We want to obtain that $D(\rho_2) <  0$ for $r \gtr 1$. One of the simplest weights is 
\beq\label{eq:EE_wg2}
\rho_2 =  r^2 \la r \ra^{k_2-2}, \quad \la r \ra \teq (1 + r^2)^{1/2},  \quad k_2 \in (3, 13),
\eeq
which is even in $r$ with $\rho_2 \asymp r^{k_2}, r > 1$. We add  $r^2$ in the weight to capture the relation between 1D integral in $r$ and 3D integral in $v$: $ \int f(|v|) r^2 d r =  C \int f(v) d v$ for radially symmetric function.

We need to choose $k_2 < 13$ 
for linear stability. See the estimate \eqref{eq:Ladd_2}.
Next, we derive the main terms of $D(\rho_2)$ for large $r$. For the first term in $D(\rho_2)$, 
using \eqref{eq:gbar_asym}, we have
\[
\bal
\bar g_r \pa_r \f{ \rho_2}{r^2} 
= r \bar g_r (k_2-2) \la r \ra^{ k_2-4}
\asymp - (k_2-2) r^2 \la r \ra^{k_2 - 5}. 
\eal
\]
For other terms in $D(\rho_2)$, we have
\[
\bal
 & \bar g_{rr} \f{\rho_2}{r^2} 
 - \f{ (\bar g_{rr} \rho_2)_{rr}  }{2} - 2 \bar f \rho_2 
 = \bar g_{rr} \la r \ra^{ k_2 - 2}
 - \bar g_{rr} \f{\rho_{2,rr}}{2} -  \bar g_{rrr} \rho_{2, r}
  - \f{1}{2}\bar g_{rrrr} \rho_2 - 2 \bar f \rho_2, \\
 & |\rho_{2, rr} - 2 \la r \ra^{ k_2 - 2 } | \les r^2 \la r \ra^{ k_2-4}, \quad 
| \rho_{2, r} | \les r  \la r \ra^{ k_2 - 2}.  \\
\eal
\]
Using these estimates and the estimate of $\bar g$ \eqref{eq:gbar_asym}, we get 
\[
 \bar g_{rr} \f{\rho_2}{r^2} 
 - \f{ (\bar g_{rr} \rho_2)_{rr}  }{2} - 2 \bar f \rho_2 
  \leq C r^2 \la r \ra^{ k_2-7}.
\]

Combining the above estimates, for some absolute constant $c , C> 0$, we yield 
\[
 D(\rho_2) \leq - c (k_2-2)  r^2 \la r \ra^{ k_2 - 5} +  C r^2 \la r \ra^{ k_2-7}.
\]

Next, we control the nonlocal terms \eqref{eq:Jrho2} using the damping term $|| f  r \la  r \ra^{ (k_2-5)/2} ||_2$ from $J(\rho_2)$. Recall the formulas of $g_r, g_{rr}$ 
\eqref{eq:BSlaw2}, \eqref{eq:BSlaw3}, which involve $A_2, A_4, B_1$ \eqref{eq:moment}.  Using the Poincare inequality, we yield 
\beq\label{eq:nloc_est1}
\bal
 \int \f{A_{k}^2}{r^{2k}} d r 
 \les \int \f{ (A_k^{\prime} )^2}{r^{2k-2}} d r 
 = \int \f{ (f r^k)^2}{ r^{2k-2}} d r = \int f^2 r^2 d r, \ k = 2, 4 . 
 \eal
\eeq

For $B_1$ and $k_2 > 6$, using Cauchy-Schwarz inequality, we get 
\beq\label{eq:nloc_est2}
|B_1(r)| \les (\int_r^{\infty} r^2 f^2 \la r \ra^{ k_2-5} d r)^{ \f{1}{2} }
 ( \int_r^{\infty}  \la r \ra^{5-k_2} d r)^{ \f{1}{2} }
 \les (\int_r^{\infty} r^2 f^2 \la r \ra^{k_2-5} d r)^{ \f{1}{2} }
 \la r \ra^{ \f{ 6-k_2}{2} }.
\eeq
For $\g \in [0, 50]$, we have
\[
 |\bar f_r (1 + r)^{\g} | + | \bar f_{rr} (1 + r)^{\g}| \les e^{-2 r^2/ 3} = \bar f^{2/3}
 \leq \bar f^{1/2}.
\]
Using \eqref{eq:BSlaw2},\eqref{eq:BSlaw3}, the above estimates of $A_2, A_4, B_1$, and the definition of $\rho_2$ \eqref{eq:EE_wg2}, we yield  
\[
\bal
 \int 2 | \f{\bar f_r g_r}{r^2} f \rho_2|  + |\bar f_{rr} g_{rr} f \rho_2| dr 
& \leq C (\int f^2 r^2 )^{1/2}
( \int f^2  r^2 \bar f)^{1/2}
 + C (\int_0^{\infty} f^2 r^2 \la r \ra^{k_2-5} d r)^{1/2} \int |f| r \bar f^{2/3}  \\
& \leq \e \int f^2 r^2 \la r \ra^{k_2-5}
+ C \e^{-1} \int f^2 r^2 \bar f ,
\eal
\] 
where we have used 
\[
\int |f| r \bar f^{2/3} 
\les (\int f^2 r^2 \bar f)^{1/2} (\int \bar f^{1/3})^{1/2} 
\les  (\int f^2 r^2 \bar f)^{1/2}.
\]

Choosing $\e = \f{c (k_2-2)}{2}$,combining the above estimates, and using 
$\bar g_{rr} \leq -c_4^* \la r \ra^{-3}$ \eqref{eq:gbar_asym} for some absolute constant $c_1 >0, c_4^*>0$, we establish 
\beq\label{eq:energy2}
\bal
J(\rho_2)
& \leq \int (- c_1 (k_2-2)  \la r \ra^{k_2-5} + C \la r \ra^{k_2-7} ) r^2 f^2 
+ \bar g_{rr} (\pa_r f)^2 \rho_2 d r \\
& \leq \int  (- c_1 (k_2-2)  \la r \ra^{-3} + C \la r \ra^{-5} )  f^2 \rho_2 
-c_4^* \la r\ra^{-3}(\pa_r f)^2 \rho_2 d r  .
 \eal
\eeq



\subsubsection{Controlling the lower order terms}\label{sec:E1_lower}


Next, we control the loss terms \eqref{eq:E1_loss_loc1}, $I_1$ \eqref{eq:E1_loss_nloc1}, $I_2$ \eqref{eq:E1_loss_nloc2}.
We introduce the following quantity related to the damping term in \eqref{eq:energy2}
\beq\label{eq:coer_norm2}
D_2 \teq \int \la r \ra^{-3}  ( f^2 +  ( \pa_r f)^2 ) \rho_2  d r
= \int   ( f^2 +  ( \pa_r f)^2 ) r^2 \la r \ra^{k_2- 5} d r . 
\eeq




Using the normalization conditions \eqref{eq:normal} $\int_0^{\infty}  f r^{\b} d r =0$, we get
\[
 I_{\b} = \int_0^{R_1} f r^{\b} d r  = - \int_{R_1}^{\infty} f r^{\b} d r  ,  \quad \b = 2, 4.
\]

To control the integral $I_{\b}$, using damping term with $k_2 > 12$,
for $\b = 2, 4$, we yield 
\beq\label{eq:E1_loss_normal}
\bal
|I_{\b } | & 
\les ( \int_{R_1}^{\infty} f^2 r^2 \la r \ra^{k_2-5} )^{1/2} 
(\int_{R_1}^{\infty} r^{ 2 \b - k_2 + 3} d r)^{1/2} 
 \les |k_2 - 12|^{- \f{1}{2}} D_2^{1/2} R_1^{ \b - k_2 / 2+ 2}, 
\eal
\eeq

Combining \eqref{eq:lin_coer} with $l = k_2-5 > 7$, \eqref{eq:coer_unif3}, and the above estimates,  we establish 
\beq\label{eq:lin_coer2}
- \la \cL_1 f , f \rho \ra
\geq 
\d_{**} \int_0^{R_1} | \bar g_{rr} r^2| \mu^{-1} ( (\pa_r f )^2 + \la r \ra^2 f^2 ) d r
- C ( R_1^{-2} + R_1^{12-k_2}) |k_2-12|^{-\f{1}{2} } D_2.
\eeq

\subsection{Estimate of the additional terms}\label{sec:term_add}

It remains to estimate the terms \eqref{eq:Q_1D}
\[
\cL_{add} f  = - \bar c_{l} r \pa_r f + \bar c_{\om} f + 2 (\al -1) \bar f f 
= \cL_{\al} - \cL_1. 
\]

For a radially symmetric weight $W(v) \les |v|^2$ near $v=0$, using integration by parts and \eqref{eq:int_rad}, we get 
\[
\la \cL_{add} f , f |v|^{-2} W \ra_{\R^3} 
= 4 \pi \int \cL_{add} f \cdot f W  d r 
= 4\pi  \int \B( \bar c_l \f{  (r W)_r }{2 W} + \bar c_{\om}  + 2 (\al-1) \bar f \B) f^2  W d r.
\]

Since $\rho_r = r q \geq 0$ \eqref{eq:nota2}, from the asymptotics \eqref{eq:rho_asym1}, \eqref{eq:rho_asym3}, we get 
\[
  \f{ r \rho_r}{\rho} = \f{r^2 q}{\rho} \les_{R_1} 1 , \mathrm{\ for \ } r \leq 10 R_1, 
  \quad   \f{r^2 q}{\rho} \les_{R_1}  \f{ r^2 \cdot r^{k-2}}{r^k} \les_{R_1}  1, \mathrm{\ for \ }  r > 10 R_1.
\]




Thus, for $W = r^2 \rho$,  we yield 
\beq\label{eq:Ladd_1}
\bal
  | \f{r \pa_r W}{W} |  &= | \f{r \pa_r( r^2 \rho)}{  r^2 \rho} |\les C(R_1), \\
  \la \cL_{add} f , f \rho \ra_{\R^3} 
 & \leq 4 \pi \int ( \bar c_{\om} + 2(\al-1) \bar f +  C(R_1) (\al-1)  ) f^2 r^2 \rho 
 \leq C(R_1) |\al-1|  \int f^2 r^2 \rho.
 \eal
\eeq

For $W = \rho_2$, the estimate is similar
\[
  \f{ \pa_r (r  \rho_2) }{ 2 \rho_2 }
   = \f{ 1 }{ 2 } \B( 3 + (k_2-2) \f{r^2}{1+r^2} \B)
   \leq \f{k_2 + 1}{2}.
\]

It follows 
\beq\label{eq:Ladd_2}
\int \cL_{add} f \cdot f \rho_2 d r
\leq \int  ( \f{k_2 + 1}{2} \bar c_l + \bar c_{\om} + 2 (\al- 1) \bar f ) f^2 \rho_2 .
\eeq

\subsection{Summarizing the linear stability estimate}\label{sec:lin_sum}

We determine the weights $\rho, \rho_2$ as follows 
\beq\label{eq:wg_para}
k = \f{5}{2} ,\quad 
\rho_2 = r^2 \la r \ra^{k_2-2}, \quad k_2 = \f{25}{2}
\eeq
where $k>2$ is the parameter for the weight $\rho$ \eqref{eq:wg_ansatz}. 
From Lemma \ref{lem:mono}, \eqref{eq:ASS}, we have $\bar c_{\om}  / \bar c_l = -7$. Using \eqref{eq:lin_coer2} and the above estimates, we establish the following estimates for $\cL_{\al} = \cL_1 + \cL_{add}$ 
\[
\bal
 \la \cL_{\al} f , f \rho \ra_{\R^3}
& \leq - \d_{**} \int_0^{R_1}  |\bar g_{ rr} r^2|  (  (\pa_r f)^2 + \la r \ra^2 f^2 ) \mu^{-1} d r 
+ C(R_1) | \al-1|  \int_0^{\infty}  f^2 r^2 \rho    \\
 & \quad + C ( R_1^{ 12 - k_2} + R_1^{-2} ) D_2 , \\
\int \cL_{\al} f \cdot f \rho_2 d r
& \leq  \int ( -  c_1^* \la r \ra^{-3 } - c_2^* (\al -1)  + c_3^*  \la r \ra^{-5} ) f^2 \rho_2 d r 
- c_4^* \int \la r \ra^{-3}   (\pa_r f)^2 \rho_2 d r, 
\eal
\]
for some absolute constants $c_i^*, C > 0$. Recall  $D_2$ from \eqref{eq:coer_norm2} and $k_2 > 12, k_2 -5 > k$ from \eqref{eq:wg_para}. Using
\eqref{eq:rho_asym3} for $\rho$ and
\[
\bal
& R_1^{-2} \leq R_1^{12  - k_2}, \quad r^2 \rho \les_{R_1} 
 r^{2} \la r \ra^k \les_{R_1} 
r^2 \la r \ra^{k_2 -5} =  \rho_2 \la r \ra^{-3}, \quad \int f^2 \rho_2 \la r\ra^{-3} \les D_2
\eal
\]
we can simplify the first estimate as follows 
\[
 \la \cL_{\al} f , f \rho \ra_{\R^3}
 \leq - \d_{**} \int_0^{R_1}  | \bar g_{ rr} r^2 | (  (\pa_r f )^2 + \la r \ra^2 f^2 ) \mu^{-1} d r 
 + ( c_6^*  R_1^{ 12 - k_2}  + C(R_1) |\al-1| ) D_2  .
\]

We want to pick $K_1, R_0^*, R_1, \al$ in order such that $ K_1  \la \cL f , f \rho \ra_{\R^3} +  \la \cL f, f \rho_2 \ra $ is coercive. 
Firstly, there exists some absolute constant $R_0^* >1$ independent of $\al$, such that for $ r > R_0^*$, we have
\[
 -  c_1^* \la r \ra^{-3 }   + c_3^*  \la r \ra^{-5} \leq - \f{1}{2} c_1^* \la r \ra^{-3 } ,
 \quad r > R_0^*.
\]

We pick $K_1$ large enough such that 
\[
c_3^*  \la r \ra^{-5} \rho_2 
= c_3^* r^2 \la r \ra^{k_2 - 7}
\leq \f{1}{2} K_1  \d_{**} | \bar g_{ rr} r^2| \mu^{-1} ,\quad r \in [0, R_0^*].
\]
Choosing $R_1 > R_0^*$ large enough and then $ 0 < \al^*(R_1)-1$ small enough, we obtain 
\[
K_1  ( c_6^*  R_1^{ 12 - k_2}  + C(R_1) |\al-1| ) < \min(c_1^*, c_4^*) / 10,
\]
for any $\al \in (1, \al^*(R_1))$, 
and further establish
\[
\bal
 K_1 \la \cL_{\al} f , f \rho \ra_{ \R^3}
& + \int \cL_{\al} f \cdot f \rho_2 d r
 \leq - \f{\d_{**} K_1}{2}  \int_0^{R_1} ( \bar g_{ rr} r^2) (  (\pa_r f )^2 + r^2 f^2 ) \mu^{-1} d r 
  \\
& \quad - \f{1}{4} \B( \int \B(  c_1^* \la r \ra^{-3 } + c_2^* (\al-1) \B) f^2 \rho_2 d r 
 + c_4^* \int \la r \ra^{-3}   (\pa_r f)^2 \rho_2 d r \B). 
 \eal
 \]

We define the energy associated with the above forms 
\beq\label{eq:energy_L2}
 E_2 \teq K_1 \la f, f \rho \ra_{\R^3} + \int f^2 \rho_2 d r.
\eeq

We fix the absolute constants $R_1, K_1$ in the rest of the paper and do not track their dependencies. The above estimate show that we have a spectral gap $|\al-1|$ for the linearized operator. Moreover, since the above estimate is uniform for small $|\al-1|$, choosing $\al=1$ and rewriting the integral in $\R^3$ \eqref{eq:int_rad}, we prove Theorem \ref{thm:coer_poly} with a weight $W = K_1 \rho(v) + \f{1}{4\pi} |v|^{-2} \rho_2(v)$.

Using the notations $D_2$ \eqref{eq:coer_norm2} and $E_2$ , we establish 
\beq\label{eq:EE_lin_L2}
 \la \cL_{\al} f , f W \ra_{\R^3} \leq - c_1 D_2 - c_2 |\al-1| E_2 ,
\eeq
for some absolute constant $c_1, c_2$ independent of $\al$ and $|\al-1|$ small enough.

\section{Linear stability for general soft potentials}\label{sec:lin_nonrad}

For very soft potential $\g \in [-3, -2)$ without radially symmetry, the linear stability analysis builds on the coercive estimates developed in \cite{carrapatoso2017landau}. We consider solution $f$ to \eqref{eq:LC_a} even in $v_i$, which is preserved, such that $f$ is orthogonal to $v_i, i=1,2,3$. 

\subsection{Stability estimates of $\cL_1$ from \cite{carrapatoso2017landau}}
Recall from Lemma \ref{lem:mono} that $\s_{\g} = |\bar c_{\om}(\g) / \bar c_l(\g)|  > 5 $ for $\g \in [-3, -2)$. We introduce some weight and functional spaces $X, Y, Z$
used in \cite{carrapatoso2017landau} 
\beq\label{eq:norm_gam}
\bal
 & m_{\g} = \la v \ra^{k_{\g} / 2 } , 
 \ k_{\g} =   2  + \s_{\g}  , \ \s_{\g} = | \f{ \bar c_{\om}(\g) }{  \bar c_l(\g) }| ,
\quad  P_v \xi \teq ( \xi \cdot \f{v}{|v|} ) \f{v}{|v|} , \quad  ||f ||_{L^2(m)}  \teq || f m ||_{L^2},  \\
& 
|| f||_{H^1_*( m)}
 \teq  
 || f m \la v \ra^{\g/2} ||_{L^2}
 + || \la v \ra^{\g/2} m P_v \na f ||_{L^2}
 + || \la v \ra^{\g/2+1} m ( I - P_v) \na f ||_{L^2} , \\
 & || f||_{H_*^{-1}(m)}
 \teq \sup_{ || g ||_{H_*^1(m)}\leq 1 } \la m f , m g \ra_{L^2},
 \quad  X = L^2( m_{\g}), \quad Y = H^1_*( m_{\g}), \quad Z = H^{-1}_*( m_{\g}) .
 \eal
\eeq


For the stability of $\cL_a$ in $L^2(m_{\g})$, the weight cannot grow too fast and we need $ k_{\g} < 2 \s_{\g}  -3$. See \eqref{eq:Ladd_gam} and discussion therein. 
Meanwhile, to control the moment  $\int f |v|^2 d v$ using $|| f||_X$, we need $k_{\g} > 7$. 
The estimate $\s_{\g} > 5$ in  Lemma \ref{lem:mono} is crucial for the existence of $k_{\g} \in (7, 2 \s_{\g} -3)$.

Let $S_{\cL}$ be the semigroup generated by the operator $\cL_1 f= Q(\mu, f ) + Q(f, \mu)$. 
Denote $\ker(\cL) = \mathrm{span} \{  \mu, v_i \mu, |v|^2  \mu \}$, $\Pi$ the projection such that $f - \Pi f \in \ker(\cL)$ and $\Pi f$ is orthogonal to $1, v_i , |v|^2$. Also, consider a class of weights $\Theta$ for $m_0, m_1$
\[
 \Theta_{m_1, m_0, l }  =  \la t \ra^{- (l_1 - l) / |\g| }, 
 \quad m_i = \la v \ra^{l_i}, \ i = 0,  1, \  l_0 < l_1 , \ l \in (l_0, l_1).
\]

The following estimates for $S_{\cL} \Pi$ were established in Theorem 3.5 \cite{carrapatoso2017landau}.
\begin{prop}\label{prop:decay}
Let $m_i = \la v \ra^{l_i}, \f{\g+3}{2} < l_0 < l_1$. For $l \in (l_0, l_1)$, there holds 
\[
|| S_{\cL}(t) \Pi ||_{ L^2(m_1) \to L^2(m_0) }
\les_{l, \g} \Theta_{m_1, m_0, l}(t) 
\les_{l, \g} \la t \ra^{- (l_1 - l) / |\g| }.
\]

\end{prop}


The following coercive estimates were established in Proposition 3.6 and Corollary 3.7 \cite{carrapatoso2017landau}.
\begin{prop}\label{prop:coer_gam}
Let $m_{\g} = \la v \ra^{k_{\g}/2}, X, Y, Z$ be defined in \eqref{eq:norm_gam} with $k_{\g}>7$. 
Denote 
\beq\label{eq:inner_time}
\Xi(f, g) \teq \int_0^{\infty} \la S_{\cL}(\tau) \Pi f,  S_{\cL}(\tau) \Pi g \ra_{L^2} d \tau, 
\quad  ||| f |||_X^2 \teq \eta || f||_X^2 + \Xi(f, f),  \quad \eta > 0,
\eeq
and  $\la \cdot , \cdot \ra_X, \llangle \cdot , \cdot \rrangle_X$ as the inner products associated with the norms $|| \cdot ||_X^2, |||  \cdot |||_X^2$, respectively. There exists $\eta > 0$ such that the norm $ ||| \cdot |||_X^2$ is equivalent to $|| \cdot ||_X^2$ on $\Pi X$. Moreover, 
we have
\beq\label{eq:coer_gam}
\bal
 & \llangle \cL_1  f ,  f \rrangle_X \les - || f ||_Y^2 , \quad \forall f \in \Pi X,  \\
 &  t \to || S_{\cL}(t) \Pi||_{Y \to L^2} 
 || S_{\cL}(t) \Pi||_{Z \to L^2 }  \in L^1, 
 \quad  |\Xi(f, g)| \les || f ||_Y  || g ||_Z.
 \eal
\eeq

\end{prop}

 Recall $\cL_a = \cL_1 + \cL_{add}$ \eqref{eq:lin}. Denote by $\cR$ the terms other than $\cL_1$ in \eqref{eq:lin_terms}
\beq\label{eq:rem_gam}
\bal
\cR_{\al} & \teq \cL_{\al} - \cL_1 + N(f) + N(\bar f)   = \cL_{add} + N(f) + N(\bar f) , \\
 \cL_{add} &=  - \bar c_l v \cdot \na f + \bar c_{\om} f
+ (\al - 1) ( c(\bar f) f + c(f ) \bar f).
\eal
\eeq

From the even symmetries of $f$ in $v_i$ and the normalization conditions \eqref{eq:normal}, 
$f, \cL_1 f, (\cL_1 + \cR_{\al}) f, \cR_{\al} f$ are orthogonal to $1, v_i, |v|^2$ for all $t > 0$. In particular $\Pi h = h$ for $h =f, \cL_1 f,  \cR_{\al} f $. Using Proposition \ref{prop:coer_gam}, we obtain 
\beq\label{eq:rem_est1}
\f{1}{2} \f{d}{dt} ||| f |||_X^2 = 
\llangle (\cL_1 + \cR_{\al}) f , f \rrangle_X 
\leq - c || f ||_Y^2 + \eta \langle \cR_{\al} f , f \ra
+ \Xi( \cR_{\al} f, f)
\eeq
We will estimate $\cR_{\al}$ perturbatively. Below, we estimate $\cL_{add}$ in $\cR_{\al}$.

\subsection{Estimates of the additional terms}

We have the following estimates for $c(f)$ in \eqref{eq:LC}.

\begin{lem}\label{lem:prod2}
For $\g \in [-3, -2)$ and $l > 7$, denote $m = \la v \ra^{ l+\g} $. We have 
\[
\bal
 |  \la c(f) g, h \la v \ra^{l}  \ra_{\R^3} | & \les_{\g, l} 
|| f \la v \ra^{l/2} ||_{L^2} 
 \min\B(  || g ||_M ||  h ||_M, 
 || g \la v \ra^{l/2}||_{L^2} ( || h ||_{L^2} + || h||_{L^{\infty}} + || h \la v \ra^{l/2 + \g} ||_{L^2} ) \B)
 ,  \\
  || p ||_M  & \teq  || p m^{1/2} ||_{L^2} +  || (\na p) m^{1/2} ||_{L^2} .  
 \eal
\]
\end{lem}

\begin{proof}

Below, we drop the dependence of norms on $\R^3$. 

Recall the formula of $c(f), a_{ij}$ \eqref{eq:LC}, \eqref{eq:LC_a}. A direct calculation yields 
\[
 c(f) = c_{\g} |v|^{\g} \ast f, \  \g \in (-3, -2), \quad c(f) = f, \ \g = -3.
\]
We introduce $p_2$ below, and for $ \g \in [-3, -2), l > 7$, we have 
\beq\label{eq:nonest_expo}
 p_2 = \f{l+ \g}{2}, \quad 
- 2 \g < l ,\quad  \f{l}{4} < p_2 .
\eeq

Firstly, using Sobolev embedding inequality, we have 
\[
 || f \la v \ra^{p_2} ||_{L^6} 
 \les || \na ( f \la v \ra^{p_2}  ) ||_{L^2}
 \les || \na f \cdot \la v \ra^{p_2}||_{L^2} 
 + || f \cdot \la v \ra^{p_2}||_{L^2}  =  || f||_M .
\]

Since $ || f \la v \ra^{p_2} ||_{L^2} \leq ||f||_M$, we yield 
\beq\label{eq:non_interp}
 || f \la v \ra^{p_2} ||_{L^p}  \les || f||_M, \quad p \in [2, 6]. 
\eeq

For $\g = -3$, since $c(f) =  f $ and \eqref{eq:nonest_expo}, using H\"older's inequality and 
\eqref{eq:non_interp}, we obtain 
\[
 |  \la c(f) g, h \la v \ra^{l}  \ra_{\R^3} |
 \les || f \la v \ra^{ l/2 } ||_{L^2}
  || g \la v \ra^{l/4} ||_{L^4}     || h \la v \ra^{l/4} ||_{L^4} 
  \les  || f \la v \ra^{ l/2 } ||_{L^2}  || g||_M || h||_M,
\]
which provides the first estimate in the Lemma. The second estimate is trivial.

For $ \g \in (-3, -2)$, we partition the integral in $c(f)$ as follows 
\[
 c(f) = C \int_{|v- w| \leq \la v \ra / 4}  f(w) |v-w|^{\g} d w 
 + C \int_{|v-w| > \la v \ra/4} f(w) |v-w|^{\g} d w  \teq c_1(f) + c_2(f).
\]

For $c_2(f)$, using Cauchy-Schwarz inequality and $l > 7$ we yield 
\[
\bal
|c_2(f)| & \les \la v \ra^{\g} \int |f(w)| d w 
\les \la v \ra^{\g} || f \la v\ra^{l/2} ||_2 , \\
 | \la c_2(f) g , h \la v \ra^l  \ra | & \les || f \la v\ra^{ \f{l}{2} } ||_2  || g h \la v \ra^{l+ \g} ||_{L^1}  
  \les || f \la v\ra^{ \f{l}{2} } ||_2 \min( || g||_M || h||_M , 
 || g \la v \ra^{ \f{l}{2} } ||_2 || h \la v \ra^{ \f{l}{2} +  \g} ||_{2} ) .
 \eal
\]

Next, we estimate $c_1(f)$. 
Since $|v- w| \leq \la v \ra / 4$ implies $ \la v \ra \asymp \la w \ra$, we get 
\[
\la v \ra^{l/2} |c_1(f) | 
\les_l \int_{|v-w| \leq \la v \ra/4} |f(w) | \la w \ra^{l/2} |v-w|^{\g} d w
\les_l  (|f(\cdot) | \la \cdot \ra^{l/2}  \ast |\cdot|^{\g})(v).
\]

Using Hardy-Littlewood-Sobolev inequality, we derive 
\[
|| \la v \ra^{l/2} |c_1(f) |  ||_{L^q} 
\les_{\g}   || |f   \la v \ra^{l/2}  ||_{L^2}, \quad 
\f{1}{q} = \f{1}{2} + \f{|\g|}{3} - 1 = \f{|\g|}{3} - \f{1}{2}. 
\]

Denote by $q^{\prime}$ the conjugate exponent 
\[
q^{\prime} = \f{1}{1-1/q} =
\f{1}{ 3/2 - |\g|/ 3 } = \f{6}{ 9 - 2 |\g|}  \in (1,2]
\]
for $\g \in [-3, -2)$.  Using H\"older's inequality, \eqref{eq:non_interp}, $2 q^{\prime} \in [2, 4]$, and \eqref{eq:nonest_expo}, we establish 
\[
\bal
&|\la c_1(f) g , h  \la v \ra^{l} \ra |
\les  || c_1(f) \la v \ra^{l/2} ||_{L^q} 
|| g h \la v \ra^{l/2} ||_{L^{q^{\prime}}} 
\les  || f \la v \ra^{l/2}||_{L^2} \cdot I, \quad I \teq  || g h \la v \ra^{l/2} ||_{L^{q^{\prime}}} , \\
& 
I \les  || g \la v\ra^{l/4} ||_{L^{ 2q^{\prime}}} 
  || h \la v\ra^{l/4} ||_{L^{ 2q^{\prime}}}  
  \les || g||_M || h ||_M, \quad I \les || g \la v \ra^{l/2} ||_{L^2}( || h||_{L^2} + || h||_{L^{\infty}} ).
\eal
\]
where we have used $ \f{1}{q^{\prime}} - \f{1}{2} \in [0, 1/2]$ in the last inequality. Combining the above estimates, we conclude the proof.
\end{proof}

Using Lemma \ref{lem:prod2} with $l = k_{\g}$, the norms $X, Y, Z$ \eqref{eq:norm_gam}, and duality, we obtain 
\beq\label{eq:cf_Z}
|| p||_M \leq || p ||_Y, \ \forall p \in Y, \quad  || c(f) g ||_Z \les || f \la v \ra^{k_{\g}/2} ||_2 || g||_Y  \les || f ||_X || g ||_Y.
\eeq


\subsubsection{Estimates of $\cL_{add}$ }\label{sec:Ladd_gam}

Recall $\cL_{add}$  \eqref{eq:rem_gam}, \eqref{eq:rem_est1}. We first estimate the term in $ \eta \la \cdot , \cdot \ra_X$ \eqref{eq:rem_est1}. 

The estimates of the scaling terms $-\bar c_l v \cdot \na f + \bar c_{\om} f $ are similar to \eqref{eq:Ladd_2} except that we need to perform integration by parts in $\R^3$. To estimate 
$(\al-1) \la \bar c(f) f + c(f) \bar f, f m_{\g}^2 \ra $, we use Lemma \ref{lem:prod2} with $l = k_{\g}, m_{\g} = \la v \ra^{k_{\g} /2}, || p ||_M \leq || p ||_Y$.  We establish 
\beq\label{eq:Ladd_gam}
\la \cL_{add} f , f m_{\g}^2 \ra
\leq \int_{\R^3}  ( \f{k_{\g}  + 3 }{2} \bar c_l + \bar c_{\om}  ) f^2 m^2_{\g} d v 
+ C|\al-1|   || f||_Y ( || f||_Y + || f m_{\g} ||_{L^2} ) .
\eeq

Using $ \bar c_{\om} / \bar c_l = -\s_{\g}$ and \eqref{eq:norm_gam}, we get 
\[
\f{k_{\g}  + 3 }{2} \bar c_l + \bar c_{\om} 
= (\f{k_{\g}  + 3 }{2} - \s_{\g}) \bar c_l =  \f{\s_{\g}-5}{2} \bar c_l = - c_{\g}  (\al-1),  
\]
for some $c_{\g} > 0$. Note that $|| f m_{\g} ||_{L^2} = || f ||_X $. Combining the above estimates, using the Cauchy-Schwarz inequality, and by requiring $\al>1, |\al-1|$ small,  we establish 
\beq\label{eq:EE_lin_gam0}
\eta\la \cL_{add} f, f  \ra_X \leq  
- c_{1,\g} || f ||_Y^2 - c_{2,\g} (\al-1) || f ||_X^2. 
\eeq
for some absolute constants $c_{1,\g}, c_{2, \g} > 0$. 

\vspace{0.1in}
\paragraph{\bf{Estimate $\Xi(\cL_{add}f, f) $}}

To estimate the integral of $\Xi$ in \eqref{eq:rem_est1}, \eqref{eq:inner_time}, we choose weights 
\[
 l_1 = \f{ k_{\g} + \g - 2}{2}, 
l_1^{\prime} = \f{ k_{\g} + \g }{2} , l_0 = \f{\g+3}{2} + \e,  l_* = \f{\g+3}{2} + 2 \e , \quad
m_1 = \la v \ra^{l_1}, 
m_1^{\prime} = \la v \ra^{l_1^{\prime}}, 
m_0 = \la v \ra^{l_0}, 
\]
Since $k_{\g} > 7, |\g| \leq 3$ and $\b=  \f{l_1 + l_1^{\prime } - 2 l_* }{  |\g| } =\f{k_{\g} - 4-4 \e}{|\g|}  > 1$ for small $\e>0$.
Applying Proposition \ref{prop:decay} to $(m_0, m_1), (m_0, m_1^{\prime})$, and choosing $0< \e < \f{ |\g|}{8} (\f{k_{\g}-4}{|\g|}-1)$, we obtain
\[
|\Xi(p, q)|=
 \B|\int_0^{ \infty} \la S(\tau) \Pi p , S(\tau) \Pi q \ra d \tau \B|
 \les  || p  ||_{L^2(m_1)} || q  ||_{L^2(m_1^{\prime})} \int_0^{\tau} \la \tau \ra^{  \f{l_1 + l_1^{\prime} - 2 l_*}{|\g|}  } d \tau 
 \les || p  ||_{L^2(m_1)} || q  ||_{L^2(m_1^{\prime})}.
\]

Applying the above estimate to $p = -\bar c_l v \cdot \na f + \bar c_{\om} f , q = f$, we obtain 
\[
|\Xi( p, q  )| 
\les  || p m_1||_{L^2} || q m_1^{\prime} ||_{L^2}
\les |\al-1| \cdot || ( |\na f| + f) \la v \ra^{ ( k_{\g}+\g)/2} ||_{L^2}
|| f \la v \ra^{ ( k_{\g}+\g)/2} ||_{L^2} \les |\al-1| \cdot || f||_Y^2.
\]

Applying the second estimate in \eqref{eq:coer_gam} in Proposition \ref{prop:coer_gam} 
to $ p  = (\al-1) ( c(\bar f) f +  (c(f) \bar f)$, and using \eqref{eq:cf_Z}, we obtain 
\[
 | \Xi(p, f)| \les || p ||_Z || f ||_Y 
 \les |\al-1|( || c(\bar f) f ||_Z + || c(f) \bar f||_Z )  || f ||_Y
\les (\al-1) ( || f ||_X + || f||_Y ) || f ||_Y.
\]

Combining the above estimates and \eqref{eq:EE_lin_gam0} and using Cauchy-Schwarz inequality, we establish 
\beq\label{eq:EE_lin_gam}
 | \Xi(\cL_{add} f, f)| + \eta \la \cL_{add} f , f \ra
 \leq (- c_{1, \g} +  C|\al-1| ) || f ||_Y^2 - c_{2, \g}|\al-1| \cdot || f||_X^2.
\eeq
for absolute constants $c_{1,\g}, c_{2,\g}, C$  independent of $\al$.

\section{Nonlinear stability and finite time blowup}\label{sec:non}

In this section, we close the nonlinear estimates and establish finite time blowup.



\subsection{Estimates of nonlinear terms}

The nonlinear estimates are relatively standard based on functional inequalities. Below, we estimate the following nonlinear terms \eqref{eq:lin_terms} in order
\[
 - c_l v \cdot \na f + c_{\om} f  , \quad 
Q(f, f) +(\al-1) c(f) \cdot f .
\]

Using  Lemma \ref{lem:prod2} with $(f, g) = (f, f), (\bar f, f), (\bar f, f)$,
$h = |v|^i \la v \ra^{-l} $ 
, $7 < l < \min( k_2-2, k_{\g}), i =0,2$, 
and $ h \in L^2 \cap L^{\infty}, || h \la v \ra^{l/2 +\g} ||_2 \les 1 $, we obtain 
\[
 \B|\int \B( c(f) f + c(\bar f)f + c(f) \bar f \B) | v|^i d v \B|
+  \B|\int  f  | v |^i d v \B| \les   || f \la v \ra^{ l/2 } ||_{L^2}
+ || f \la v \ra^{ l/2 } ||_{L^2}^2 . 
\]
From the normalization conditions \eqref{eq:normal},\eqref{eq:lin_terms}, 
in the case of Coulomb with radially symmetry, using $|| f \la v \ra^{ l/2 } ||_{L^2}
\les E_2^{1/2} $ \eqref{eq:energy_L2}, we further get 
\beq\label{eq:EE_non_clcw}
\bal
| C_j(f + \bar f) - C_j(\bar f)| 
\les (E_2^{1/2} + E_2), \ j = 1, 2, \quad |c_l | , |c_{\om}| \les |\al-1| \cdot (E_2^{1/2} + E_2).
\eal
\eeq

Following the integration by parts argument in \eqref{eq:Ladd_1}, \eqref{eq:Ladd_2}, we get 
\beq\label{eq:EE_non_clcw2}
 |   \la - c_l v \cdot \na f + c_{\om} f, f \rho \ra |
 + |\la - c_l v \cdot \na f + c_{\om} f, f \cdot  ( |v|^{-2} \rho_2) \ra | \les |\al-1| (E_2^{1/2} + E_2) E_2. 
\eeq

Similarly, for general cases, we get 
\beq\label{eq:EE_non_clcw_gam}
|c_l| + |c_{\om}| \les (\al -1) \cdot ( || f ||_X + ||f ||_X^2).
\eeq
The estimates of the operator are similar to those in Section \ref{sec:Ladd_gam}
\beq\label{eq:EE_non_clcw_gam2}
| \la - c_l v \cdot \na f + c_{\om} f, f m_{\g}^2 \ra  | 
+ | \Xi( - c_l v \cdot \na f + c_{\om} f, f  ) | 
\les (|c_l| + |c_{\om}|) ( || f ||_Y^2 + || f||_X^2 ) .
\eeq

\vspace{0.1in}
\paragraph{\bf{Coulomb potential with radially symmetric solution}}
The estimates are straightforward, and we only give a sketch. Let $\rho_2$ be the weight defined in \eqref{eq:wg_para}. It satisfies $\rho_2 \gtr \la v \ra^{7} |v|^2$. Using the formulas \eqref{eq:moment} and Cauchy-Schwarz inequality, we  obtain 
\[
|A_2(r) |\les \min(1, r^{5/2}) E_2^{1/2}, \quad  
|A_4(r)| \les \min(1, r^{9/2}) E_2^{1/2}, \quad 
|B_1(r)| \les \min(1, r^{(3-k_2)/2} ) E_2^{1/2}.
\]
which along with \eqref{eq:BSlaw2}, \eqref{eq:BSlaw3} yield the pointwise estimates for $g$ similar to \eqref{eq:gbar_asym} 
\[
|g_r| \les \f{r}{1 + r} E_2^{1/2}, \quad | g_{rr} | \les \la r \ra^{-3} E_2^{1/2},
\quad | g_{rrr} | \les \min(r, r^{-4}) E_2^{1/2}, 
\quad |g_{rrrr}|\les \min(1, r^{-5}) E_2^{1/2} .
\]

Since $r^2 \rho \les_{R_1} \rho_2$, following the estimates of the nonlocal terms in Section \ref{sec:E2_lin}, we derive 
\beq\label{eq:EE_non}
\bal
 | \int N(f) f r^2 \rho | 
 + | \int N(f) f \rho_2 |
 & \leq 
 E_2^{1/2} \B(  \int ( f^2 + (\pa_r f)^2) \la r \ra^{-3} \rho_2  
 +  |\al-1|  \int f^2 \rho_2 d r \B) \\
 & \leq E_2^{1/2} ( D_2 
 +  |\al-1|  E_2   ),
 \eal
\eeq
where  $D_2, E_2$ are defined in \eqref{eq:coer_norm}, \eqref{eq:energy_L2}.

\vspace{0.1in}
\paragraph{\bf{General cases}}

In general cases for soft potentials without radial symmetry, the estimates can be obtained using functional inequalities. We use the following estimate for the collision operator from Corollary 4.4 \cite{carrapatoso2017landau}.

\begin{prop}\label{lem:prod1}
Consider any weight function $m \gtr \la v \ra^{l}, l > 7$. we have 
\[
\la Q(f, g), h \ra_X
\les ( || f ||_X || g ||_Y + ||f ||_Y || g ||_X) || h||_Y,
\quad  || Q(f, g) ||_Z \les || f ||_X || g ||_Y + || f ||_Y || g||_X.
\]
\end{prop}


Recall $m_{\g} = \la v \ra^{k_{\g}/2}, k_{\g} > 7$ and the norm $X, Y$ from \eqref{eq:norm_gam}. Using Lemmas \ref{lem:prod1}, \ref{lem:prod2} with $(f, g, h) = (f,f,f)$ and $l=k_{\g}$, 
$\la p, q \ra_X = \pa p, q m_{\g}^2 \ra$ \eqref{eq:norm_gam}, we obtain 
\beq\label{eq:non_gam}
| \la Q(f, f) + (\al-1) c(f) f , f m^2_{\g} \ra |
\les || f ||_X || f ||_Y^2 , 
\eeq

To estimate the $\Xi$ term in \eqref{eq:inner_time}, we use \eqref{eq:coer_gam} in Proposition \ref{prop:coer_gam}, Proposition \ref{lem:prod1} for $Q$, and \eqref{eq:cf_Z} for $c(f) f$ to obtain 
\beq\label{eq:non_gam2}
| \Xi ( Q(f, f) + (\al-1) c(f) f , f  )
\les || Q(f, f) + (\al-1) c(f) f||_Z || f ||Y
\les || f ||_X || f ||_Y^2 .
\eeq

\subsection{Estimate the residual error}
Recall the residual error \eqref{eq:lin_terms}
\[
N(\bar f) = - ( \bar c_l + c_l) v \cdot \na  \bar f + ( \bar c_{\om} + c_{\om} )\bar f + (\al-1) c(\bar f) \bar f. 
\]

Using $|\bar c_l| \les |\al-1|,  |\bar c_{\om}| \les |\al-1|$ \eqref{eq:ASS}, 
the estimate \eqref{eq:EE_non_clcw} for $c_l, c_{\om}$, and the fast decay of $\bar f$, 
\[
|\pa_r^i \bar f | \les_i \la r \ra^{i} \bar f, \quad  
\la r \ra^{\g} \bar f \les_{\g} e^{-r^2/2}, \quad  r^2 \rho \les \rho_2, 
\]
in the case of Coulomb potential with weights \eqref{eq:wg_para}, we obtain $c(\bar f) = \bar f$ and
\beq\label{eq:EE_err}
\bal
\int | N(\bar f) f| (\rho r^2 + \rho_2) d r 
& \les   \int | N(\bar f) f| \rho_2 d r  
  \les     ( \int f^2 \la r \ra^{-3} \rho_2 )^{1/2}
(\int N(\bar f)^2 \la r \ra^3 \rho_2 )^{1/2}  \\
 & \les |\al-1| (1 + E_2^{1/2} + E_2)  D_2^{1/2}
 \les |\al-1| (1  + E_2)  D_2^{1/2}\\
\eal
\eeq

For general cases, using Lemma \ref{lem:prod2} with $(f, g, h) = (\bar f, \bar f, f)$ and the 
decay of $\bar f$, we obtain 
\beq\label{eq:EE_err_gam}
| \la N(\bar f), f \rho_{\g} \ra | 
\les |\al-1| (1 + || f||_X^2 ) || f ||_Y.
\eeq

\subsection{Nonlinear stability and finite time blowup}\label{sec:blowup}

We impose a weak bootstrap assumption 
\beq\label{eq:boot_weak}
  E_2(t) \leq 1 .
\eeq
\paragraph{\bf{Coulomb case with radial symmetry}}
Combining the estimates of linear, nonlinear, and error terms \eqref{eq:EE_lin_L2}, 
 \eqref{eq:EE_non_clcw2}, \eqref{eq:EE_non}, \eqref{eq:EE_err}, and using $E_2 \leq 1$ \eqref{eq:boot_weak}, we establish 
\[
\bal
 \f{d}{dt} E_2 
 & \leq  
 ( \lam_1  E_2^{1/2} - \lam_2) 
 (D_2 + |\al -1| E_2)
 + \lam_3 |\al-1|  D_2^{1/2} 
  \eal
\]
for some absolute constant $\lam_1, \lam_2 ,\lam_3> 0$ independent of $\al$, and $\al>1$ with $|\al-1|$ small enough. Using $|\al-1| D_2^{1/2} \les \e D_2 + \e^{-1} |\al-1|^2  $ with $\e$  small, we can further obtain 
\beq\label{eq:boot0}
 \f{d}{dt} E_2 
\leq 
( \lam_1  E_2^{1/2} - \td  \lam_2) 
 (D_2 + |\al-1| E_2)
 + \td \lam_3 |\al-1|^2  .
\eeq
for some $\td \lam_2, \td \lam_3 > 0$ independent of $\al$. Now, we set the bootstrap threshold 
\beq\label{eq:boot_thres}
E_* = \f{2 \td \lam_3}{ \td \lam_2} |\al-1| .
\eeq

Then, there exists an absolute constant $\al_* > 1$ with $\al_* - 1$ small enough such that for $0< \al-1 < \al_* - 1$, 
the right hand side of \eqref{eq:boot0} is negative for $E = E_*$ and $E_* < 1$. As a result, for $E(0)< E_*$ and  all $ t> 0$, we prove 
\beq\label{eq:boot}
E_2(t) < E_* < 1,
\eeq
and the weak bootstrap assumption \eqref{eq:boot_weak} can be continued.

From \eqref{eq:EE_non_clcw}, we can further require that $|a_*-1|$ is small enough such that 
\eqref{eq:boot} implies 
\beq\label{eq:boot_expo}
|c_l | \leq \f{\bar c_l}{2} , \  |c_{\om}| \leq \f{\bar c_{\om}}{2} ,
\quad \f{|c_{\om} + \bar c_{\om}| } { |c_{l} + \bar c_{l}| }> \f{\s_{\g} + 5 }{2} > 5,
\quad c_l + \bar c_l \asymp |\al-1|, \ c_{\om} + \bar c_{\om} \asymp - |\al-1|,
\eeq
with implicit constants uniformly in $t$, where $\s_{\g} = | \f{ \bar c_{\om }} {\bar c_l}| > 5$ \eqref{eq:norm_gam} \eqref{eq:ASS0}.  
By truncating the approximate profile  $\bar f$, we can construct initial data $f_{in} = \bar f+ f \in C_c^{\infty}$ or $\bar f+ f \in C^{\infty}$ decaying exponentially fast, and it is even in $v_i$ and satisfies $E_2(f) < E_*$ and the conditions \eqref{eq:normal0} for $f_{in}$.
Using the rescaling relation \eqref{eq:dyn1} between the solution $f_{phy}$ to the physical equation \eqref{eq:LC_a} from $f_{phy}(0) = f_{in}$ and its associated solution to the rescaling equation \eqref{eq:LC_a_rescal} $F$, we obtain 
\beq\label{eq:blowup_vk}
\bal
  & ||  f_{phy}( t(\tau)) |v|^k ||_{L^1} = 
 C_{\om}(\tau)^{-1} C_l(\tau)^{k+3} || F(\tau) |v|^k ||_{L^1} \\
 = & \exp\B( \int_0^{\tau} (- c_{\om}(s) - \bar c_{\om} - (k+3) (c_l(s) + \bar c_l)) d s \B) ||  F(\tau) |v|^k ||_{L^1}
 \eal
\eeq
The estimate \eqref{eq:boot} implies $E_2( F(\tau) - \bar f ) < E_*$ for all $\tau > 0$. 
Since $\f{|c_{\om} + \bar c_{\om}| } { |c_{l} + \bar c_{l}| }> \f{1}{2}(\s_{\g} + 5) > 5$  \eqref{eq:boot_expo}, the mass $(k=0)$ and energy $(k=2)$ blows up as $\tau \to \infty, t(\tau) \to  T = t(\infty) < +\infty$. Note that \eqref{eq:boot_expo} implies that $C_{\om} C_l^{-3-\g}$ in \eqref{eq:dyn2} decays exponentially in $\tau$ and thus $t(\infty) < +\infty$.


Since a sufficiently smooth perturbation $f$ with $ f(v) = \la v \ra^{- l}, l \geq 7$ for
sufficiently large $v$ is in the energy class  $E_2(f) < +\infty$ \eqref{eq:energy2}, $f \in X $ \eqref{eq:norm_gam}, 
we can also construct initial data $f_{in} = \bar f + f$ with small  $f$ that leads to blowup.


\vspace{0.1in}
\paragraph{\bf{General case}}

Similarly, we first impose a weak bootstrap assumption $|| f ||_X \leq 1$.
Combining the estimates for linear, nonlinear, and error terms \eqref{eq:EE_lin_gam}, \eqref{eq:non_gam}, \eqref{eq:non_gam2},
\eqref{eq:EE_non_clcw_gam}, \eqref{eq:EE_non_clcw_gam2}, \eqref{eq:EE_err_gam}, and then following \eqref{eq:boot0} and using the equivalence between $||| \cdot |||_X$ and $|| \cdot ||_X$ from Proposition \ref{prop:coer_gam}, we obtain 
\beq\label{eq:boot_gam}
  \f{d}{dt} ||| f |||_X^2 \leq ( c_{1,\g}  ||| f |||_X + c_{4,\g} |\al-1| - c_{2, \g})
  ( || f ||_Y^2 + |\al -1| \cdot ||| f |||_X^2 ) + c_{3,\g} (\al - 1)^2. 
\eeq
for some $c_{i, \g} > 0$ independent of $\al $. We choose a threshold $E_{*, \g} = e_{ \g} |\al -1| < 1$  with $1 < \al < \al_{*, \g}$ for some absolute constant $e_{\g}$ and $\al_{*,\g} > 1, |\al_{*,\g}-1|$ small to obtain nonlinear stability estimates $ ||| f |||_X^2 < E_{*, \g} $. The remaining steps of finite time blowup are the same as the previous case.


\begin{remark}\label{rem:asymp}
By performing higher order stability estimates and using the time-differentiation argument
\cite{chen2019finite,chen2019finite2,chen2021HL,chen2020slightly,chen2021regularity}, one can prove that the blowup is asymptotically self-similar and $|| f||_{L^{\infty}}$ blows up in finite time. In the time-differentiation argument for convergence, due to the spectral gap $|\al-1|$ in \eqref{eq:boot0}, \eqref{eq:boot_gam}, one can obtain exponential convergence of $f + \bar f $ to the \textit{exact} self-similar profile of \eqref{eq:LC_a}, which is close to the Maxwellian $ \mu$.

\end{remark}

\vspace{0.2in}
\noindent
{\bf Acknowledgments.} 
JC is grateful to Vlad Vicol for introducing the Landau equation and some stimulating discussion at the early stage of the work.

\appendix


\section{Derivations for radially symmetric functions}\label{app:rad}

In this appendix, we derive some formulas 
 for radially symmetric functions: $f(v) = F(|v|)$ for some function $F$. 
Denote $r = |v|$. 
To simplify the notations, for any radially symmetric function $f(v) = F(|v|), v\in \R^3$ for some $F$, we will write $f(r)$ for $F(r)$. 
We recall the notations at the beginning of Section \ref{sec:lin}.



\subsection{Derivations for the Landau equation with radial symmetry}


Denote 
\[
g = (-\D)^{-2} f = \D^{-2} f, 
\quad g_1 = \D^{-1} f .
\]

For any radially symmetric function $p$, we have 
\beq\label{eq:deri}
\bal
\pa_i p &= \f{v_i}{r} p_r, \quad \pa_{ij} p = ( \f{\d_{ij}}{r} - \f{v_i v_j}{r^3} ) p_r + \f{v_i v_j}{r^2} p_{rr} , \quad  \D p = (\pa_{rr} + \f{2}{r} \pa_r) p.
\eal
\eeq

Using the above identities, we yield 
\[
\pa_{ij}(-\D)^{-2} f(v) \cdot \pa_{ij} f 
= \pa_{ij } g \pa_{ij} f = \B( ( \f{\d_{ij}}{r} - \f{v_i v_j}{r^3} ) g_r + \f{v_i v_j}{r^2} g_{rr}  \B) \B( ( \f{\d_{ij}}{r} - \f{v_i v_j}{r^3} ) f_r + \f{v_i v_j}{r^2} f_{rr} \B) ,
\]
where the index $i$ or $j$ sums over $1,2,3$ if it appear twice. Since
\[
 ( \f{\d_{ij}}{r} - \f{v_i v_j}{r^3} )  \f{v_i v_j}{r^2} 
 = \f{v_i v_i}{r^3} - \f{ v_i v_i v_j v_j}{r^5} = \f{1}{r} - \f{1}{r} = 0,
\]
the cross terms $f_r g_{rr}, g_r f_rr$ in the above expansion vanish. We get
\[
\bal
\pa_{ij}(-\D)^{-2} f(v)  \cdot \pa_{ij} f 
&= \sum_{i, j} ( \f{\d_{ij}}{r} - \f{v_i v_j}{r^3} )^2 f_r g_r + \f{v_i^2 v_j^2}{r^4} f_{rr} g_{rr} \\
& = ( \f{3}{r^2} - \f{ 2 v_i v_i}{r^4} +  \f{v_i v_i v_j v_j}{r^6}    ) f_r g_r 
+ f_{rr} g_{rr} 
 = \f{2}{r^2} f_r g_r + f_{rr} g_{rr} .
\eal
\]

Hence, for radially symmetric function $f$, we can rewrite $Q(f, f)$ \eqref{eq:Q} and \eqref{eq:LC_a} as \eqref{eq:Q_1D}.

\subsection{Formulas for the nonlocal terms}

In this section, we derive the formula for $g = \D^{-2} f$. We assume that $f$ has a decay rate $ |f| \les r^{-l}$ for $r \geq 1$ and some $l > 2$.
Denote 
\beq\label{eq:moment}
A_k(r) = \int_0^r f(s) s^k ds, \quad B_k(r) = \int_r^{\infty} f(s) s^k ds.
\eeq

Clearly, we have 
\beq\label{eq:moment_deri}
\pa_{ r} A_k =  f(r) r^k, \quad \pa_r B_k = - f(r) r^k.
\eeq

We first derive $g_1 \teq \D^{-1} f$. From $\D g_1 = f$, we have 
\[
(\pa_{rr} + \f{2}{r} \pa_r ) g_1 = f, \quad  \pa_r( r^2 \pa_r g_1) = r^2 f.
\]

Integrating both sides from $0$ to $r$ yields 
\beq\label{eq:BSlaw_deri1}
r^2 \pa_r g_1 =  \int_0^r s^2 f(s) ds, \quad \pa_r g_1 = \f{1}{r^2} \int_0^r f(s ) s^2 ds.
\eeq


Since $|f| \les r^{- l}, l > 2$, integrating both sides from $r$ to $\infty$, we derive
\[
\bal
 g_1(r)  & = \D^{-1} f = - \int_r^{\infty} \f{1}{z^2} \int_0^z f(s) s^2 ds dz 
= - \int_0^{\infty} f(s) s^2 \int_{ z \geq \max( s, r) }  \f{1}{z^2} dz ds \\
&= -\int_0^{ \infty} f(s) \f{s^2}{  \max(s, r)  } ds
= - \f{1}{r}\int_0^r f(s) s^2 ds - \int_r^{\infty} f(s) s  ds  =  -\f{1}{r} A_2(r) - B_1(r) .
\eal
\]

It is easy to verify that $\D g_1 = f$. The above formula shows that if $f(r)$ has a decay rate $r^{-l}$ for large $r$ with $ l > 2$, then $g_1(r)$ has a decay rate $r^{-\b}, \b =\min(1, l-2)$. 

Next, we derive $ g = \D^{-1} g_1$. Using the same derivation as \eqref{eq:BSlaw_deri1}, we obtain 
\[
 \pa_r g = \f{1}{r^2} \int_0^r g_1(s) s^2 ds.
\]

In $Q(f,f)$ \eqref{eq:Q_1D}, we only need the formula for $\pa_r g$. Combining the formula for $g_1$ and $\pa_r g$, we obtain 
\[
\bal
\pa_r g(r) &=  - \f{1}{r^2} \int_0^r  z^2 \B( \f{1}{z}\int_0^z f(s) s^2 + \int_z^{\infty} f(s) s ds
 \B)  d z \teq I + II.
\eal
\]
For $I$, using the operators $A_k, B_k$ \eqref{eq:moment}, we can simplify it as follows 
\[
I = - \f{1}{r^2} \int_0^r f(s) s^2 \int_s^r z dz  ds
= - \f{1}{r^2} \int_0^r f(s) \f{r^2 - s^2}{2} s^2 ds
 = -\f{1}{2} A_2(r) + \f{1}{2} \f{A_4(r)}{r^2}  .
\]
For $II$, we have 
\[
\bal
II &= - \f{1}{r^2} \int_0^r z^2 \int_z^{\infty} f(s) s ds
= -\f{1}{r^2} \int_0^{\infty} f(s) s \int_0^{\min(s, r)} z^2 dz  ds \\
& = - \f{1}{r^2} \int_0^{\infty} f(s) s \cdot \f{1}{3} (\min(s, r))^3 ds
 = -\f{1}{ 3r^2} \B( \int_0^r f(s) s^4 + r^3 \int_r^{\infty}  f(s) s ds \B)  = -\f{A_4(r)}{3 r^2} - \f{r}{3} B_1(r).
\eal
\]

Combining the above formulas, we obtain 
\beq\label{eq:BSlaw2}
g_r(r) = I + II = - \f{1}{2} A_2 + \f{1}{6} \f{A_4}{r^2} - \f{r}{3} B_1.
\eeq

Using \eqref{eq:moment_deri}, we can obtain the formulas for $\pa_r^k g, k=2,3,4$
\beq\label{eq:BSlaw3}
\bal
\pa_r^2 g  & = - \f{1}{2} r^2 f + \f{1}{6} \f{ f r^4}{r^2} - \f{1}{3} \f{A_4}{r^3} - \f{B_1}{3} +\f{ fr^2}{3} =  -\f{1}{3} \f{A_4}{r^3} - \f{B_1}{3}, \\
\pa_r^3 g & = \f{A_4}{r^4} - \f{1}{3} \f{f r^4}{r^3} + \f{1}{3} fr = \f{1}{r^4} A_4(r), \\
\pa_r^4 g & =  \f{ fr^4}{r^4} - \f{4 A_4}{r^5} = f - \f{4 A_4}{r^5}.
\eal
\eeq
Using the above formulas, we determine \eqref{eq:Q_1D}. For $f \geq 0$, from \eqref{eq:BSlaw2},\eqref{eq:BSlaw3}, we obtain that 
\[
g_r \leq -\f{1}{2} A_2 + \f{1}{6} \f{A_4}{r^2}
= \int_0^r f(s) ( -\f{1}{2} s^2 + \f{1}{6} \f{s^4}{r^2}) ds \leq 0, \quad 
g_{rr} \leq 0, \quad g_{rrr} \geq 0.
\]

\bibliographystyle{plain}
\bibliography{selfsimilar}

\end{document}